\documentclass[12pt,oneside]{amsart}

      \usepackage{amssymb}
\usepackage[margin=1in]{geometry}  
\usepackage[usenames,dvipsnames]{xcolor}
\usepackage{ulem}       
\usepackage{amsmath, amsthm, amssymb}
\usepackage{xspace}
\usepackage{mathrsfs}
\usepackage[pdftex,bookmarks,colorlinks,breaklinks]{hyperref}  
\hypersetup{linkcolor=blue,citecolor=red,filecolor=dullmagenta,urlcolor=darkblue} 
 \usepackage{graphicx}  
 \usepackage{color}	
 \usepackage{mathabx}
\newtheorem{theorem}{Theorem}[section]

\newtheorem{lemma}[theorem]{Lemma}

\newtheorem{example}[theorem]{Example}

\newcommand{\scr}{\mathscr}   
   
      \makeatletter
      \def\@setcopyright{}
      \def\serieslogo@{}
      \makeatother

\begin{document}

   \author{Amin Bahmanian}
   \address{Department of Mathematics and Statistics, 221 Parker Hall\\
 Auburn University, Auburn, AL USA   36849-5310}
  
   \title[Detachments of Amalgamated 3-uniform Hypergraphs]{Detachments of Amalgamated 3-uniform Hypergraphs  : Factorization Consequences}

   \begin{abstract}
A \textit{detachment} of a hypergraph $\scr F$ is a hypergraph obtained from $\scr F$ by splitting some or all of its vertices into more than one vertex. \textit{Amalgamating} a hypergraph $\scr G$ can be thought of as taking $\scr G$, partitioning its vertices, then for each element of the partition squashing the vertices to form a single vertex in the amalgamated hypergraph $\scr F$. 
In this paper we use Nash-Williams lemma on laminar families to prove a detachment theorem for amalgamated 3-uniform hypergraphs, which yields a substantial generalization of previous amalgamation theorems by Hilton, Rodger and Nash-Williams.  

To demonstrate the power of our detachment theorem, we show that the complete 3-uniform $n$-partite multi-hypergraph $\lambda K_{m_1,\ldots,m_n}^{3}$ can be expressed as the union $\scr G_1\cup \ldots \cup\scr G_k$ of $k$ edge-disjoint factors, where for $i=1,\ldots, k$, $\scr G_i$ is $r_i$-regular, if and only if:
\begin{itemize}
\item [(i)] $m_i=m_j:=m$ for all $1\leq i,j\leq k$, 
\item [(ii)] $3\divides r_imn$ for each $i$, $1\leq i\leq k$, and  
\item [(iii)] $\sum_{i=1}^{k} r_i=\lambda \binom{n-1}{2}m^2$. 
\end{itemize}
   \end{abstract}

   \date{\today}

   \maketitle

\section {Introduction}
A \textit{detachment} of a hypergraph $\scr F$ is, informally speaking, a hypergraph obtained from $\scr F$ by splitting some or all of its vertices into more than one vertex. If $\scr G$ is a detachment of $\scr F$, then $\scr F$ is an \textit{amalgamation} of $\scr G$. Amalgamating $\scr G$, intuitively speaking, can be thought of as taking $\scr G$, partitioning its vertices, then for each element of the partition squashing the vertices to form a single vertex in the amalgamated hypergraph $\scr F$. We shall give more precise definition for amalgamation and detachment in Section \ref{notation}. 

Hilton \cite{H2} used amalgamation to decompose complete graphs into Hamiltonian cycles, obtaining a new proof of Walecki's result \cite{L}.
Hilton and Rodger \cite{HR} produced new proofs of Laskar and Auerbach's results on Hamiltonian decomposition of the complete multipartite graphs. Buchanan \cite{B} used amalgamations to prove that for any $2$-factor $U$ of $K_n$, $n$ odd, $K_n-E(U)$ admits a Hamiltonian decomposition. Rodger and Leach \cite{LR1} solved the corresponding existence problem for complete bipartite graphs, and obtained a solution for complete multipartite graphs when $U$ has no small cycles \cite{LR2, LR3}.

Perhaps the most interesting use of amalgamations has been to prove embedding results; see, for example \cite{AndHil1, AndHil2, HJRW, MatJohns, Nash87, RW}. Detachments of graphs have also been studied in \cite{BeJacksonJor,JacksonJor}, generalizing some results of Nash-Williams \cite{NashW85, NashW}. 
For a survey about the method of amalgamation and embedding partial edge-colorings we refer the reader to \cite{AndRod1}.  

Most of the results in graph amalgamation have used edge-coloring techniques due to de Werra \cite{deW71, deW71-2, deW75, deW75-2}, however Nash-Williams \cite{Nash87} proved a lemma (see Lemma \ref{laminarlem} in Section \ref{hyp1proof} below) 
to generalize theorems of Hilton and Rodger.  In this paper we apply Nash-Williams technique to produce a general detachment theorem for 3-uniform hypergraphs (see Theorem \ref{hyp1}). This result is not only a substantial generalization of previous amalgamation theorems, but also yields several consequences on factorizations of complete 3-uniform multipartite (multi)hypergraphs.  To demonstrate the power of our detachment theorem, we show that the complete 3-uniform $n$-partite multi-hypergraph $\lambda K_{m_1,\ldots,m_n}^{3}$ can be expressed as the union $\scr G_1\cup \ldots \cup\scr G_k$ of $k$ edge-disjoint factors, where for $i=1,\ldots, k$, $\scr G_i$ is $r_i$-regular, if and only if:
\begin{itemize}
\item [(i)] $m_i=m_j:=m$ for all $1\leq i,j\leq k$, 
\item [(ii)] $3\divides r_imn$ for each $i$, $1\leq i\leq k$, and  
\item [(iii)] $\sum_{i=1}^{k} r_i=\lambda \binom{n-1}{2}m^2$. 
\end{itemize}

It is expected that Theorem \ref{hyp1} can be used to provide conditions under which one can embed a $k$-edge-colored complete 3-uniform hypergraph  $K_{n}^{3}$ into an edge-colored  $K_{n+m}^{3}$ such that $i^{th}$ color class of  $K_{n+m}^{3}$ induces an $r_i$-factor for $i=1,\ldots,k$. However obtaining such results will require more advanced edge-coloring techniques and it will be much more complicated than for companion results for simple graphs, with a complete solution unlikely to be found in the near future (see \cite{BahRod4Emb2}).

In connection with Kirkman's famous Fifteen Schoolgirls Problem \cite{Kirk1847}, Sylvester remarked in 1850 that the complete 3-uniform hypergraph with 15 vertices, is 1-factorizable. Several generalizations of this problem were solved during the last 70 years (see for example \cite{Pelt, RayWil73, Baran75, Baran79}).  It was Baranyai, who died tragically in his youth, who settled this 120-year-old problem (1-factorization of complete uniform hypergraphs) ingeniously \cite{Baran75, Baran79}.  

Baranyai's proof actually yields a method for constructing a 1-factorization recursively. However, this approach would not be very efficient and its complexity is exponential \cite{JungGNA}. 
Baranyi's original theorem was spurred by  Peltesohn's result \cite{Pelt} which was a direct construction, and it was polynomial time to implement. 
Brouwer and Schrijver gave an elegant proof for 1-factorizations of  the complete uniform hypergraph for which the algorithm is more efficient \cite{BrouSchrij}. Our construction leads to an algorithm similar to that of Brouwer and Schrijver. This is discussed briefly in Section \ref{algocons}, but for more details we refer the reader to \cite{BahThesis}.

Notation and more precise definitions will be given in Section \ref{notation}. Any undefined term may be found in \cite{HypBerge}. In Section \ref{mhyp1}, we state our main result and we postpone its proof to Section \ref{hyp1proof}. In Section \ref{factorizationcor}, we exhibit some applications of our result by providing several factorization theorems for 3-uniform (multi)hypergraphs. 
The key idea used in proving the main theorem is short and is given in \ref{constG}. The rest of Section \ref{hyp1proof} is devoted to the verification of all conditions in Theorem \ref{hyp1}. 

\section{ Notation and More Precise Definitions}\label{notation}
In this paper $\mathbb{R}$ denotes the set of real numbers and $\mathbb{N}$ denotes the set of positive integers.
If $f$ is a function from a set $X$ into a set $Y$ and $y\in Y$, then $f^{-1}(y)$ denotes the set $\{x\in X:f(x)=y\}$, and $f^{-1}[y]$ denotes $f^{-1}(y)\backslash\{y\}$. If $x, y$ are real numbers, then $\lfloor x \rfloor$ and $\lceil x \rceil$ denote the integers such that $x-1<\lfloor x \rfloor \leq x \leq \lceil x \rceil < x+1$, and $x\approx y$ means $\lfloor y \rfloor \leq x\leq \lceil y \rceil$. 

For the purpose of this paper, a \textit {hypergraph} $\scr{G}$ is an ordered quintuple $(V(\scr{G}),  E(\scr{G}), H(\scr{G}),$ $\psi, \phi)$ where $V(\scr{G}),  E(\scr{G}),  H(\scr{G})$ are disjoint finite sets, $\psi:H(\scr{G}) \rightarrow V(\scr G)$ is a function and  $\phi: H(\scr G) \rightarrow E(\scr G)$ is a surjection.  
  Elements of $V(\scr{G}), E(\scr{G}), H(\scr{G})$ are called \textit{vertices}, \textit{hyperedges} and \textit{hinges} of $\scr G$, respectively. 
A vertex $v$ and hinge $h$ are said to be \textit{incident} with each other if $\psi(h)=v$. A hyperedge $e$ and  hinge $h$ are said to be \textit{incident} with each other if $\phi(h)=e$. A hinge $h$ is said to \textit{attach} the hyperedge $\phi(h)$ to the vertex $\psi(h)$. In this manner, the vertex $\phi(h)$ and the hyperedge $\psi(h)$ are said to be \textit{incident} with each other. If $e\in E(\scr G)$, and $e$ is incident with $n$ hinges $h_1, \ldots, h_n$ for some $n\in \mathbb N$, then the hyperedge $e$ is said to \textit{join} (not necessarily distinct) vertices $\psi(h_1), \ldots, \psi(h_n)$.  If $v\in V(\scr G)$, then the number of hinges incident with $v$ is  called the \textit{degree} of $v$ and is denoted by $d_{\scr G}(v)$.  

The number of vertices incident with a hyperedge $e$, denoted by $|e|$, is called the \textit{size} of $e$. If $|e|=1$ then $e$ is called a \textit{loop}. 
If for all hyperedges $e$ of $\scr G$, $|e|\leq 2$ and $|\phi^{-1}(e)|=2$, then $\scr G$ is a \textit{graph}.  
If $n>1$ and $e_1,\ldots, e_n$ are $n$ distinct hyperedges of $\scr G$, incident with the same set of vertices, then $e_1,\ldots, e_n$ is said to be \textit{multiple hyperedges}.  A \textit{multi-hypergraph} is a hypergraph with  multiple hyperedges.  

Thus a hypergraph, in the sense of our definition is a generalization of a finite hypergraph as usually defined, but for convenience, we imagine each hyperedge of a hypergraph to be attached to the vertices which it joins by in-between objects called hinges. In fact if for every edge $e$, $|e|=|\phi^{-1}(e)|$, then our definition is essentially the same as the usual definition. One can think of a hypergraph as a bipartite multigraph, where $E$ forms one class, $V$ forms other class, and the hinges $H$ form the edges. 
A hypergraph may be drawn as a set of points representing the vertices. An edge is represented by a simple closed curve enclosing its incident vertices. A hinge is represented by a small line attached to the vertex incident with it (see Figure \ref{figure:hypexample1}). 
\begin{example}\label{hyp1ex}
\textup{
 Let $\scr F=(V,  E, H, \psi, \phi)$, with $V=\{v_i:1\leq i\leq 7\}, E=\{e_1,e_2, e_3\}, H=\{h_i: 1\leq i\leq 9\}$, such that $\psi(h_1)=v_1, \psi(h_2)=\psi(h_3)=v_2, \psi(h_4)=v_3, \psi(h_5)=\psi(h_6)=\psi(h_7)=v_4, \psi(h_8)=v_5, \psi(h_9)=v_6$ and $\phi(h_1)=e_1, \phi(h_2)=\phi(h_3)=\phi(h_4)=\phi(h_5)=\phi(h_6)=e_2, \phi(h_7)=\phi(h_8)=\phi(h_9)=e_3$.  Moreover $|e_1|=1,|e_2|=|e_3|=3$, and $d(v_1)=d(v_3)=d(v_5)=d(v_6)=1, d(v_2)=4, d(v_4)=3, d(v_7)=0$. 
\begin{figure}[htbp]
\begin{center}
\scalebox{.75}
{ \includegraphics {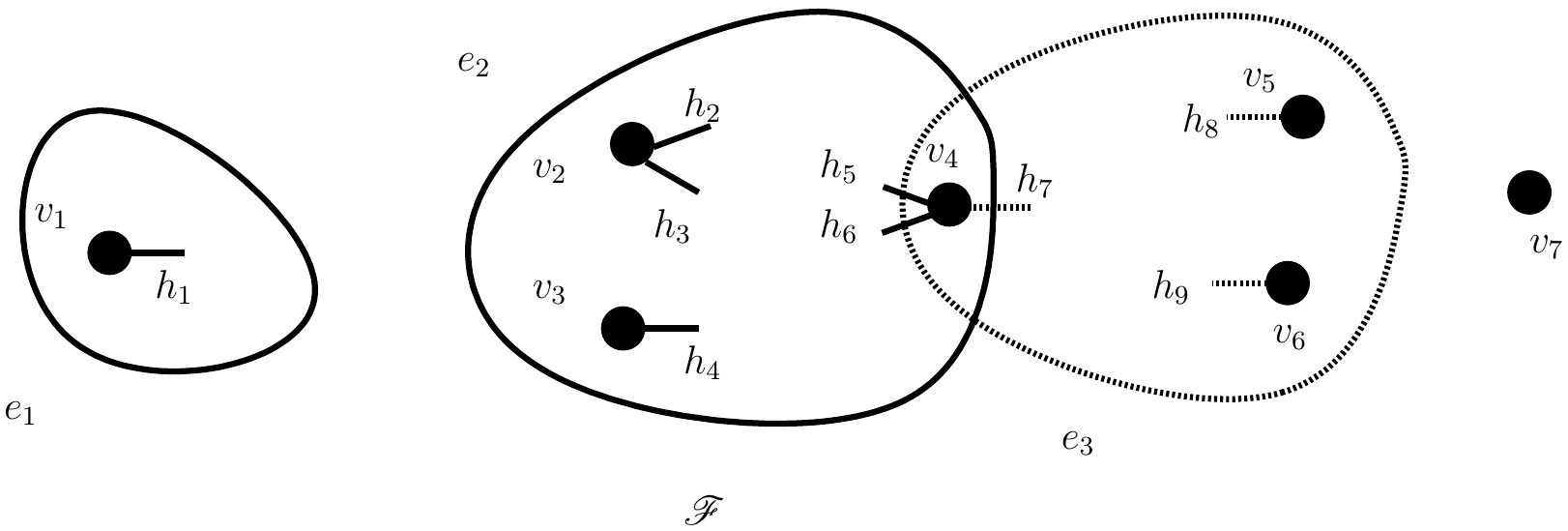} }
\caption{Representation of a hypergraph $\scr F$ }
\label{figure:hypexample1}
\end{center}
\end{figure} 
}\end{example}

Throughout this paper, the letters $\scr F$ and $\scr G$ denote hypergraphs (possibly with loops and multiple hyperedges). The set of hinges of $\scr G$ which are incident with a vertex $v$ (a hyperedge $e$), is denoted by $H(\scr G, v)$ ($H(\scr G, e)$, respectively).  Thus if $e\in E(\scr G)$, then $H(\scr G, e)=\phi^{-1}(e)$.
If $v\in V(\scr G)$, then $H(\scr G, v)=\psi^{-1}(v)$, and $|H(\scr G, v)|$ is the degree $d(v)$ of $v$. If $S$ is a subset of $V(\scr G)$ or $E(\scr G)$, then $H(\scr G, S)$ denotes the set of those hinges of $\scr G$ which are incident with an element of $S$. If $S_1 \subset V(\scr G)$ and  $S_2 \subset E(\scr G)$, then $H(\scr G, S_1, S_2)$ denotes $H(\scr G, S_1) \cap H(\scr G, S_2)$.  If $v \in V(\scr G)$ and  $S \subset E(\scr G)$, then $H(\scr G, v, S)$ denotes $H(\scr G, \{v\}, S)$.
To avoid ambiguity, subscripts may be used to indicate the hypergraph in which hypergraph-theoretic notation should be interpreted --- for example, $d_{\scr G}(v)$.

Let $\scr G$ be a hypergraph in which each hyperedge is incident with exactly three hinges.   If $u,v,w$ are three (not necessarily distinct) vertices of $\scr G$, then $\nabla (u,v, w)$ denotes the set of hyperedges which are incident with $u,v,w$. For each hyperedge $e$ incident with three hinges $h_1,h_2,h_3$ there are three possibilities (see Figure \ref{figure:edge3types}): 
\begin{itemize} 
\item [(i)] $e$ is incident with exactly one vertex $u$. In this case $u$ is incident with $h_1,h_2,h_3$. We denote $\nabla (u,u,u)$ by $\nabla (u^3)$. 
\item [(ii)] $e$ is incident with exactly two distinct vertices $u, v$. In this case one of the vertices, say $u$ is incident with two hinges, say $h_1,h_2$ and $v$ is incident with $h_3$. We denote $\nabla (u,u,v)$ by $\nabla(u^2, v)$.  
\item [(iii)] $e$ is incident with three distinct vertices $u,v$ and $w$. 
\end{itemize} 
For \textit{multiplicity} we use $m(.)$ rather than $| \nabla(.)|$.
\begin{figure}[htbp]
\begin{center}
\scalebox{.75}
{ \includegraphics {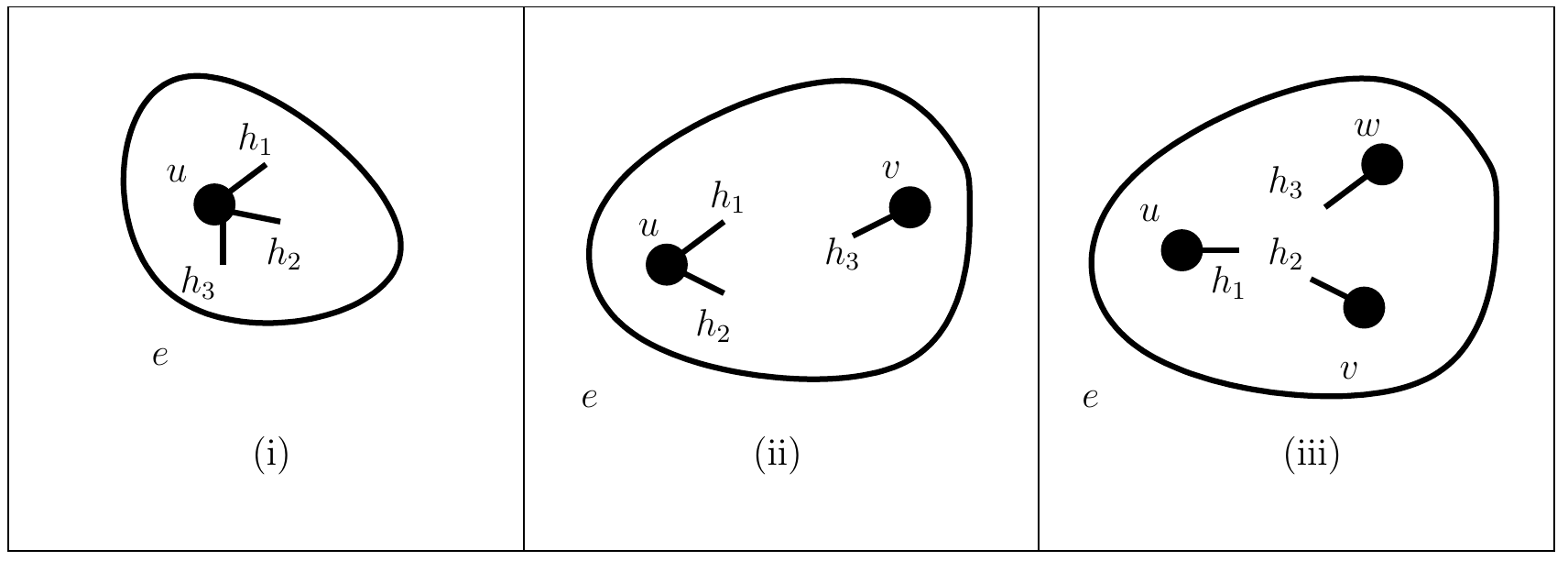} }
\caption{ The three types of edges in a hypergraph $\scr G$ in which $|H(\scr G,e)|=3$ for every edge $e$}
\label{figure:edge3types}
\end{center}
\end{figure} 
A hypergraph $\scr G$ is said to be \textit{$k$-uniform} if $|e|=|H(\scr G,e)|=k$ for each $e\in E(\scr G)$. A $k$-uniform hypergraph with $n$ vertices is said to be \textit{complete}, denoted by $K_n^k$, if every $k$ distinct vertices are incident within one edge. A $3$-uniform hypergraph with vertex partition $\{V_1,\ldots,V_n\}$ with $|V_i|=m_i$ for $i=1,\ldots,n$, is said to be (i) \textit{n-partite}, if every edge is incident with at most one vertex of each part, 
and (ii) \textit {complete n-partite}, denoted by $K_{m_1,\ldots, m_n}^3$, if it is $n$-partite and every three distinct vertices from three different parts are incident.

If we replace every hyperedge of $\scr G$ by $\lambda$ ($\geq 2$) multiple hyperedges, then we denote the new (multi) hypergraph by $\lambda \scr G$. A \textit{k-hyperedge-coloring} of $\scr G$ is a mapping $K: E(\scr G)\rightarrow C$, where $C$ is a set of $k$ \textit{colors} (often we use $C=\{1,\ldots,k\}$), and the hyperedges of one color form a \textit{color class}. The sub-hypergraph of $\scr G$ induced by the color class $j$ is denoted by $\scr G(j)$.

A hypergraph $\scr G$ is said to be (i) \textit{regular} if there is an integer $d$ such that every vertex has degree $d$, and (ii) \textit{k-regular} if every vertex has degree $k$. A \textit{factor} of $\scr G$ is a regular spanning sub-hypergraph of $\scr G$. A \textit{$k$-factor} is a $k$-regular factor. A \textit{factorization} is a decomposition (partition) of $E(\scr G)$ into factor(s). Let $r_1,\ldots, r_k$ be (not necessarily distinct) positive integers. An \textit{$(r_1,\ldots, r_k)$-factorization} is a factorization in which there is one $r_i$-factor for $i=1,\ldots, k$. An $(r)$-factorization is called simply an \textit{$r$-factorization}. 
 A hypergraph $\scr G$ is said to be \textit{factorizable} if it has a factorization. The definition for \textit{$k$-factorizable} and \textit{$(r_1,\ldots, r_k)$-factorizable}  hypergraphs is similar.

If $\scr F=(V,  E, H, \psi, \phi)$ is a hypergraph and $\Psi$ is a function from $V$ onto a set $W$, then we shall say that the hypergraph $\scr G=(W,  E, H, \Psi \circ \psi, \phi)$ is an \textit{amalgamation} of $\scr F$ and that $\scr F$ is a \textit{detachment} of $\scr G$. In this manner, $\Psi$ is called an \textit{amalgamation function}, and $\scr G$ is the \textit{$\Psi$-amalgamation} of $\scr F$. Associated with $\Psi$ is the \textit{number function} $g:W\rightarrow \mathbb N$ defined by $g(w)=|\Psi^{-1}(w)|$, for each $w\in W$, and we shall say that $\scr F$ is a \textit{g-detachment} of $\scr G$. Intuitively speaking, a $g$-detachment of $\scr G$ is obtained by splitting each $u\in V(\scr G)$ into $g(u)$ vertices. 
 Thus $\scr F$ and $\scr G$ have the same hyperedges and hinges, and each vertex $v$ of $\scr G$ is obtained by identifying those vertices of $\scr F$ which belong to the set $\Psi^{-1}(v)$. In this process, a hinge incident with a vertex $u$ and a hyperedge $e$ in $\scr F$ becomes incident with the vertex $\Psi(u)$ and the  edge $e$ in $\scr G$. Since two hypergraphs $\scr F$ and $\scr G$ related in the above manner have the same hyperedges, coloring the hyperedges of one of them is the same thing as coloring the hyperedges of the other. Hence an amalgamation of a hypergraph with colored hyperedges is a hypergraph with colored hyperedges.

\begin{example} 
\textup{
Let $\scr F$ be the hypergraph of Example \ref{hyp1ex}. Let $\Psi:V\rightarrow \{w_1,w_2,w_3,w_4\}$ be the function with $\Psi(v_1)=\Psi(v_7)=w_1$, $\Psi(v_2)=w_2$, $\Psi(v_3)=\Psi(v_4)=w_3$, $\Psi(v_5)=\Psi(v_6)=w_4$. The hypergraph $\scr G$ in Figure \ref{figure:amalgamhypex1} is the $\Psi$-amalgamation of $\scr F$.  
\begin{figure}[htbp]
\begin{center}
\scalebox{.75}
{ \includegraphics {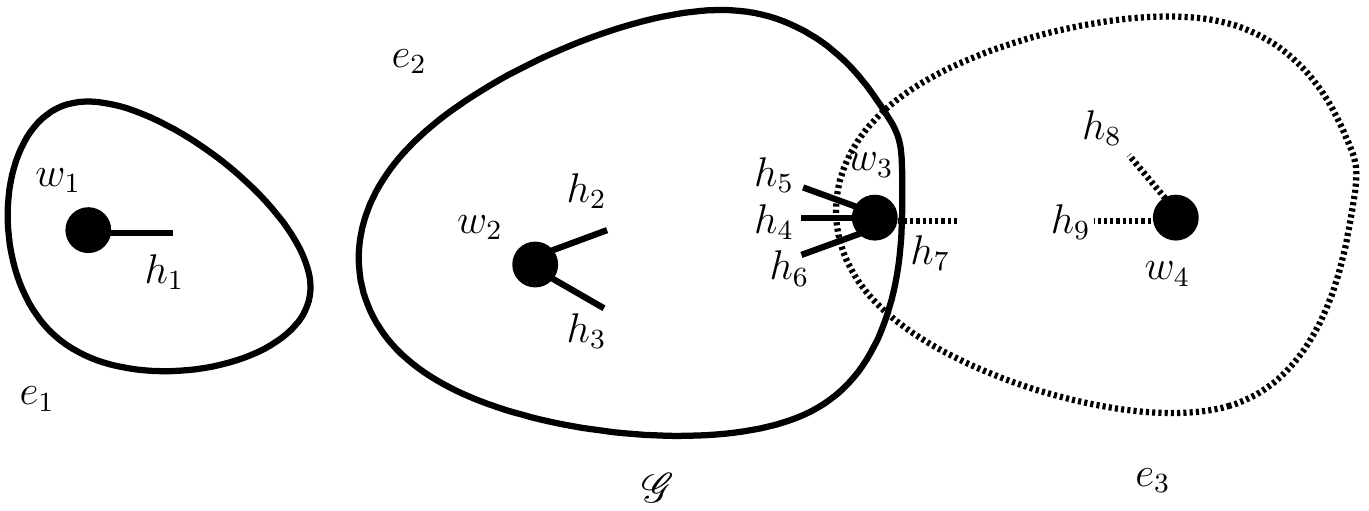} }
\caption{Amalgamation $\scr G$  of  the hypergraph $\scr F$ in Example \ref{hyp1ex}}
\label{figure:amalgamhypex1}
\end{center}
\end{figure} 
}\end{example}

\section{Statement of the Main Theorem}\label{mhyp1}
In the remainder of this paper, all hypergraphs are either 3-uniform or are amalgamations of 3-uniform hypergraphs. That is, for every hypergraph $\scr F$ we have 
\begin{equation} \label{edgeassump}
1\leq |e|\leq |H(\scr F,e)|=3  \mbox{ for every  } e \mbox{ in } \scr F. 
\end{equation}
Therefore every edge is of one the types shown in Figure \ref{figure:edge3types}.  For $g:V(\scr F)\rightarrow \mathbb{N}$, we define the symmetric function $\tilde g:V^3(\scr F)\rightarrow \mathbb{N}$ such that for distinct $x,y,z\in V(\scr F)$,  $\tilde g(x,x,x)=\binom{g(x)}{3}$, $\tilde g(x,x,y)=\binom{g(x)}{2}g(y)$, and $\tilde g(x,y,z) =g(x)g(y)g(z)$. 
Also we assume that for each $x \in V(\scr F)$, $g(x) \leq 2$ implies $m_\scr{F} (x^3) = 0$, and $g(x)=1$ implies $m_\scr F (x^2,y)=0$ for every $y\in V(\scr F)$. 
\begin{theorem}  \label {hyp1}
Let $\scr F$ be a $k$-hyperedge-colored hypergraph and let $g$ be a function from $V(\scr F)$ into $\mathbb{N}$. Then there exists a $3$-uniform $g$-detachment $\scr G$ (possibly with multiple hyperedges) of $\scr F$ with amalgamation function $\Psi:V(\scr G)\rightarrow V(\scr F)$,  $g$  being the number function associated with $\Psi$, such that $\scr G$ satisfies the following conditions:
\begin{itemize}
\item [\textup{(A1)}] $d_\scr G(u) \approx d_\scr F(x)/g(x) $ for each $x\in V(\scr F)$ and each $u\in \Psi^{-1}(x);$
\item [\textup{(A2)}] $d_{\scr G(j)}(u) \approx d_{\scr F(j)}(x)/g(x)$ for each $x\in V(\scr F)$, each $u\in \Psi^{-1}(x)$ and each $j\in \{1,\ldots,k\} ;$
\item [\textup{(A3)}] $m_\scr G(u, v, w) \approx m_\scr F(x,y,z)/\tilde g(x,y,z) $ for every $x,y,z\in V(\scr F)$  with $g(x)\geq 3$ if $x=y=z$, and $g(x)\geq 2$ if $|\{x,y,z\}|=2$, and every triple of distinct vertices $u,v,w$ with  $u\in \Psi^{-1}(x)$, $v\in \Psi^{-1}(y)$ and $w\in \Psi^{-1}(z);$
\item [\textup{(A4)}] $m_{\scr G(j)}(u, v, w) \approx m_{\scr F(j)}(x,y,z)/\tilde g(x,y,z) $ for every $x,y,z\in V(\scr F)$ with $g(x)\geq 3$ if $x=y=z$, and $g(x)\geq 2$ if $|\{x,y,z\}|=2$,  every triple of distinct vertices $u,v,w$ with  $u\in \Psi^{-1}(x)$, $v\in \Psi^{-1}(y)$ and $w\in \Psi^{-1}(z)$ and  each $j\in \{1,\ldots,k\}.$
\end{itemize}
\end{theorem}

\section{Factorization Consequences}\label{factorizationcor}
Throughout this section $n\geq 3$. It is easy to see that every factorizable hypergraph must be regular. If $\scr G$ is a $3$-uniform hypergraph with an $r$-factor, since each edge contributes 3 to the  sum of the degree of all vertices in an $r$-factor, $r | V(\scr G) |$ must be divisible by 3.

\subsection{Factorizations of $\lambda K_n^{3}$}
We first note that $\lambda K^{3}_n$ is $\lambda \binom{n-1}{2}$-regular, and $|E(\lambda K^{3}_n)|=\lambda \binom{n}{3}$. Throughout this section, $\scr F$ is a hypergraph consisting of a single vertex $x$ and $\lambda \binom{n}{3}$ loops incident with $x$, and $g:V(\scr F)\rightarrow \mathbb{N}$ is a function with $g(x)=n$. Note that $\lambda K^{3}_n$ is a $g$-detachment of $\scr F$.  

\begin{theorem} \label{lambdakn3r1rk}
$\lambda K_n^{3}$ is $(r_1,\ldots,r_k)$-factorizable if and only if  
\begin{itemize}
\item [\textup {(i)}]  $3\divides r_in$ for each $i$, $1\leq i\leq k$, and
\item [\textup {(ii)}] $\sum_{i=1}^{k} r_i=\lambda \binom{n-1}{2}$. 
\end{itemize} 
\end{theorem}
\begin{proof}
Suppose first that $\lambda K_n^{3}$ is $(r_1,\ldots,r_k)$-factorizable. The existence of each $r_i$-factor implies that  $3\divides r_in$ for each $i$, $1\leq i\leq k$. Since each $r_i$-factor is an $r_i$-regular spanning sub-hypergraph and $\lambda K^{3}_n$ is $\lambda \binom{n-1}{2}$-regular, we must have $\sum_{i=1}^{k} r_i=\lambda \binom{n-1}{2}$.  

Now assume (i)--(ii). We find a $k$-hyperedge-coloring for $\scr F$ such that 
$m_{\scr F(j)}(x^3)=r_jn/3$  for each $j\in \{1,\ldots,k\}.$
It is possible, because 
\begin{eqnarray*}
\sum_{j=1}^{k} m_{\scr F(j)}(x^3)& = & \sum_{j=1}^k \frac{r_jn}{3}=\frac{n}{3}\sum_{j=1}^k r_j\\
&=&\frac{\lambda n}{3}\binom{n-1}{2}=\lambda \binom{n}{3}=m_{\scr F}(x^3).
\end{eqnarray*}
Now by Theorem \ref{hyp1}, there exists a 3-uniform $g$-detachment $\scr G$ of $\scr F$ with $n$ vertices, say $v_1,\ldots, v_n$ such that 
by (A2) $d_{\scr G(j)}(v_i)=r_jn/n=r_j$ for each $i=1,\ldots, n$ and each $j\in \{1,\ldots,k\}$; and by (A3) $m_{\scr G}(v_r,v_s, v_t)=\lambda \binom{n}{3}/\binom{n}{3}=\lambda$ for distinct $r,s,t$, $1\leq r,s,t\leq n$. Therefore $\scr G\cong \lambda K_n^{3}$ and each color class $i$ is an $r_i$-factor for $i=1,\ldots, k$. 
\end{proof}

\subsection{Factorizations of $K_{m_1,\ldots,m_n}^{3} $}
We denote $K_{\underbrace{m,\ldots,m}_n}^{3}$ by $K_{m,\ldots,m}^{3}$ (so we don't write the under-brace when it is not ambiguous). We first note that $\lambda K_{m,\ldots,m}^{3}$ is a $\lambda \binom{n-1}{2}m^2$-regular hypergraph with $nm$ vertices and $\lambda \binom{n}{3}m^3$ edges.
Throughout this section, $\scr F =\lambda m^3K_n^{3}$ with vertex set $V(\scr F)=\{x_1,\ldots,x_n\}$, and $g:V(\scr F)\rightarrow \mathbb{N}$ is a function with $g(x_i)=m$ for $i=1,\ldots,n$. We observe that $\lambda K_{m,\ldots,m}^{3}$ is a $g$-detachment of $\scr F$.  

\begin{theorem}
$\lambda K_{m_1,\ldots,m_n}^{3}$ is $(r_1,\ldots,r_k)$-factorizable if and only if  
\begin{enumerate}
\item  [\textup {(i)}] $m_i=m_j:=m$ for $1\leq i<j\leq n$,
\item  [\textup {(ii)}] $3\divides r_imn$ for each $i$, $1\leq i\leq k$, and
\item   [\textup {(iii)}] $\sum_{i=1}^{k} r_i=\lambda\binom{n-1}{2}m^2$.  
\end{enumerate}
\end{theorem}
\begin{proof}
Suppose first that $\lambda K_{m_1,\ldots,m_n}^{3}$ is $r$-factorizable (so it is regular). Let $u$ and $v$ be two vertices from two different parts, say $p^{th}$ and $q^{th}$ parts respectively. Then we have the following sequence of equivalences:

\begin{align*}
d(u)&=d(v) &\iff\\
\sum\nolimits_{\scriptstyle 1 \leq i < j \leq n \hfill \atop  \scriptstyle i,j \ne p \hfill}  m_im_j  &=  \sum\nolimits_{\scriptstyle 1 \leq i < j \leq n \hfill \atop  \scriptstyle i,j \ne q \hfill}  m_im_j &\iff\\
m_q\sum\nolimits_{\scriptstyle 1 \leq i  \leq n \hfill \atop  \scriptstyle i \ne p,q \hfill}  m_i+\sum\nolimits_{\scriptstyle 1 \leq i < j \leq n \hfill \atop  \scriptstyle i,j \notin \{p,q\} \hfill}  m_im_j  &=
m_p\sum\nolimits_{\scriptstyle 1 \leq i  \leq n \hfill \atop  \scriptstyle i \ne p,q \hfill}  m_i+\sum\nolimits_{\scriptstyle 1 \leq i < j \leq n \hfill \atop  \scriptstyle i,j \notin \{p,q\} \hfill}  m_im_j &\iff\\
m_q\sum\nolimits_{\scriptstyle 1 \leq i  \leq n \hfill \atop  \scriptstyle i \ne p,q \hfill}  m_i&=m_p\sum\nolimits_{\scriptstyle 1 \leq i  \leq n \hfill \atop  \scriptstyle i \ne p,q \hfill}  m_i  &\iff\\
 (m_p-m_q)\sum\nolimits_{\scriptstyle 1 \leq i  \leq n \hfill \atop  \scriptstyle i \ne p,q \hfill}  m_i &=0 &\iff  \\
  m_p&=m_q:=m. &(n\geq 3)
\end{align*}

This proves (i).  The existence of each $r_i$-factor implies that  $3\divides r_imn$ for each $i$, $1\leq i\leq k$. Since each $r_i$-factor is an $r_i$-regular spanning sub-hypergraph and $K_{m,\ldots,m}^{3}$ is $\lambda \binom{n-1}{2}m^2$-regular, we must have $\sum_{i=1}^{k} r_i=\lambda \binom{n-1}{2}m^2$.  

Now assume (i)--(iii). Since $3\divides r_imn$ for each $i$, $1\leq i\leq k$ and  $\sum_{i=1}^{k} mr_i=\lambda\binom{n-1}{2}m^3$, by Theorem \ref{lambdakn3r1rk}, $\scr F$ is $(mr_1,\ldots,mr_k)$-factorizable. Therefore we can find a $k$-hyperedge-coloring for $\scr F$ such that 
$$d_{\scr F(j)}(x)=r_jm  \quad \forall j\in \{1,\ldots,k\}.$$
Now by Theorem \ref{hyp1}, there exists a 3-uniform $g$-detachment $\scr G$ of $\scr F$ with $mn$ vertices, say $x_{ij}$, $1\leq i \leq n$, $1\leq j\leq m$ ($x_{i1},\ldots,x_{im}$ are obtained by splitting $x_i$ into $m$ vertices for $i=1,\ldots,n$) such that by (A2) $d_{\scr G(t)}(x_{ij})=r_tm/m=r_t$ for each $i=1,\ldots, n$, $j=1,\ldots, m$, and each $t\in \{1,\ldots,k\}$; by (A3) $m_{\scr G}(x_{ij},x_{ij'}, x_{ij''})=0$ for $i=1\ldots,n$ and distinct $j,j',j''$, $1\leq j,j',j''\leq m$, if $m\geq 3$; by (A3) $m_{\scr G}(x_{ij},x_{ij'}, x_{i'j''})=0$ for distinct $i,i'$, $1\leq i,i'\leq n$ and distinct $j,j'$,  $1\leq j,j',j''\leq m$, if $m\geq 2$; and by (A3) $m_{\scr G}(x_{ij},x_{i'j'}, x_{i''j''})=\lambda m^3/(mmm)=\lambda$ for distinct $i,i',i''$, $1\leq i,i',i''\leq n$ and $1\leq j,j',j''\leq m$. Therefore $\scr G\cong \lambda K_{m,\ldots,m}^{3}$ and each color class $i$ is an $r_i$-factor for each $i\in \{1,\ldots,k\}$.
\end{proof}

\section{Proof of the Main Theorem} \label{hyp1proof}
Recall that $x\approx y$ means $\lfloor y \rfloor \leq x\leq \lceil y \rceil$. We observe that for $x, y \in \mathbb{R}, a, b, c \in \mathbb{Z}$, and $n\in \mathbb{N}$ (i) $a\approx x$ implies $a\in \{\lfloor x \rfloor, \lceil x \rceil \}$, (ii) $x\approx y$ implies $x/n\approx y/n$ (iii) the relation $\approx$ is transitive (but not symmetric), and (vi) $a=b-c$ and $c\approx x$, implies $a\approx b-x$. These properties of $\approx$ will be used in this section when required without further explanation.

A family $\scr A$ of sets is \textit{laminar} if, for every pair $A, B$ of sets belonging to $\scr A$, either $A\subset B$, or $B\subset A$, or $A\cap B=\varnothing$. To prove Theorem \ref{hyp1}, we need the following lemma:

\begin{lemma}\textup{(Nash-Williams \cite[Lemma 2]{Nash87})}\label{laminarlem}
If $\scr A, \scr B$ are two laminar families of subsets of a finite set $S$, and $\ell\in \mathbb N$, then there exist a subset $A$ of $S$ such that for every $P\in \scr A \cup \scr B$,  $|A\cap P|\approx |P|/\ell$. 
\end{lemma}

Let $\scr F =(V, E, H, \psi, \phi)$. Let $n = \sum\nolimits _{v \in V} (g (v) -1)$. Our proof of Theorem \ref{hyp1} consists of the following major parts. First,  in Section \ref{constG} we shall describe the construction of a sequence $\scr F_0=\scr F,\scr F_1,\dots,\scr F_n$ of hypergraphs where $\scr F_{i}$ is an amalgamation of $\scr F_{i+1}$ (so $\scr F_{i+1}$ is a detachment of $\scr F_{i}$) for $0\leq i\leq n-1$ with amalgamation function $\Phi_i$ that combines a vertex with amalgamation number 1 with one other vertex. 
To construct each $\scr F_{i+1}$ from $\scr F_i$ we will use two laminar families $\scr A_i$ and  $\scr B_i$. In Section \ref{recursionF_i} we shall observe some properties of $\scr F_{i+1}$ in terms of $\scr F_i$. 
As we will see in Section \ref{FiF}, the relations between $\scr F_{i+1}$ and $\scr F_i$ lead to conditions relating each $\scr F_i$, $1\leq i\leq n$ to the initial hypergraph $\scr F$. Finally, in Section \ref{FnF} we will show that $\scr F_n$ satisfies the conditions (A1)--(A4), so we can let $\scr G=\scr F_n$.  

\subsection{Construction of $\scr G$} \label{constG}
Initially we let $\scr F_0= \scr F$ and $g_0=g$, and we let $\Phi_0$ be the identity function from $V$ into  $V$. Now assume that $\scr F_0=(V_0, E_0, H_0, \psi_0,\phi_0),\ldots,\scr F_i=(V_i,E_i, H_i,\psi_i,\phi_i)$ and $\Phi_0,\ldots,\Phi_i$ have been defined for some $i\geq 0$. Also assume that $g_0:V_0\rightarrow\mathbb{N},\ldots, g_i:V_i\rightarrow\mathbb{N}$ have been defined  such that for each $j=0,\ldots, i$ and each $x \in V_j$, $g_j(x) \leq 2$ implies $m_{\scr{F}_j} (x^3) = 0$, and $g_j(x)=1$ implies $m_{\scr F_j} (x^2,y)=0$ for every $y\in V_j$. Let $\Psi_i=\Phi_0\ldots\Phi_i$. If $i=n$, we terminate the construction,  letting $\scr G=\scr F_n$ and  $\Psi=\Psi_n$. 

If $i< n$, we can select a vertex $\alpha$ of $\scr F_i$ such that  $g_i(\alpha)\geq 2$. As we will see, $\scr F_{i+1}$ is formed from $\scr F_i$ by detaching a vertex $v_{i+1}$ with amalgamation number 1 from $\alpha$. 

Let  $H_{ij}=H(\scr F_i(j), \alpha)$ for $j=1,\ldots, k$. If $e\in E_i$ incident with $\alpha$, we let $H_{ij}^e=H(\scr F_i(j), \alpha,e)$ for $j=1,\ldots, k$. Recall that by (\ref{edgeassump}),  $|H_{ij}^e|\leq 3$. Intuitively speaking, $H_{ij}$ is the set of all hinges which are incident with $\alpha$ and a hyperedge colored $j$, and $H_{ij}^e$ is a subset of $H_{ij}$ consisting of only those hinges incident with a single hyperedge $e$ colored $j$. 
Now let 
\begin{eqnarray}\label{lamAi}
\scr A_i & = &  \{H(\scr F_i,\alpha)\} \nonumber\\
&\bigcup & \{H_{i1},\ldots,H_{ik}\} \nonumber\\
& \bigcup & \{H_{ij}^e: e\in   \nabla(\alpha^2, y), y\in V_i, 1\leq j \leq k\}.
\end{eqnarray}
Note that 
\begin{eqnarray*}
\{H_{ij}^e: e\in   \nabla(\alpha^2, y), y\in V_i, 1\leq j \leq k\}& = &\{H_{ij}^e: e\in   \nabla(\alpha^3), 1\leq j \leq k\}\\ 
& \bigcup & \{H_{ij}^e: e\in   \nabla(\alpha^2, y), y\in V_i\backslash\{ \alpha\}, 1\leq j \leq k\}. 
\end{eqnarray*}
If $u,v\in V_i$, let $H_i^{uv}=H(\scr F_i,\nabla(\alpha, u, v))$ and $H_{ij}^{uv}= H(\scr F_i(j),\alpha, \nabla(\alpha, u, v))$ for $j=1,\ldots, k$. 
Now let
\begin{eqnarray} \label{lamBi}
\scr B_i&=&\{H_i^{uv}: u,v \in V_i\} \nonumber \\
&\bigcup &\{H_{ij}^{uv}: u,v \in V_i, 1\leq j \leq k\}.
\end{eqnarray}
It is easy to see that both $\scr A_i$ and $\scr B_i$ are laminar families of subsets of $H(\scr F_i, \alpha)$. 
Then, by Lemma \ref{laminarlem}, there exists a subset $Z_i$ of $H(\scr F_i, \alpha)$ such that 
\begin{equation}\label{lamapp1'} |Z_i\cap P|\approx |P|/g_i(\alpha),  \mbox{ for every  } P\in \scr A_i \cup \scr B_i. \end{equation}
Let $v_{i+1}$ be a vertex which does not belong to $V_i$ and let $V_{i+1}=V_i\cup \{v_{i+1}\}$.
Let $\Phi_{i+1}$ be the function from $V_{i+1}$ onto $V_i$ such that $\Phi_{i+1}(v)=v$ for every $v\in V_i$ and $\Phi_{i+1}(v_{i+1})=\alpha$. Let $\scr F_{i+1}$ be the detachment of $\scr F_i$ under $\Phi_{i+1}$ ($\scr F_i$ is the $\Phi_{i+1}$-amalgamation of $\scr F_{i+1}$) such that $V(\scr F_{i+1})=V_{i+1}$, and
 
\begin{equation}\label{hinge1} H(\scr F_{i+1},v_{i+1})=Z_i, H(\scr F_{i+1},\alpha)=H(\scr F_i, \alpha)\backslash Z_i. \end{equation}

In fact, $\scr F_{i+1}$ is obtained from $\scr F_i$ by splitting $\alpha$ into two vertices $\alpha$ and  $v_{i+1}$ in such a way that hinges which were incident with $\alpha$ in $\scr F_i$ become incident in $\scr F_{i+1}$ with $\alpha$ or $v_{i+1}$ according as they do not or do belong to $Z_i$, respectively. 
Obviously, $\Psi_i$ is an amalgamation function from $\scr F_{i+1}$ into $\scr F_i$. Let $g_{i+1}$ be the function from $V_{i+1}$ into $\mathbb N$, such that $g_{i+1}(v_{i+1})=1, g_{i+1}(\alpha)=g_i(\alpha)-1, g_{i+1}(v)=g_i(v)$ for every $v\in V_i\backslash\{\alpha\}$. This finishes the construction of $\scr F_{i+1}$. Now, we explore some relations between $\scr F_{i+1}$ and $\scr F_i$.   In the remainder of this paper, $d_i(.)$, and $m_i(.)$, $d(.)$, and $m(.)$ will denote  $d_{\scr F_i}(.)$, and $m_{\scr F_i}(.)$, $d_{\scr F}(.)$, and $m_{\scr F}(.)$, respectively.

\subsection{Relations between $\scr F_{i+1}$ and $\scr F_i$}\label{recursionF_i}
The hypergraph $\scr F_{i+1}$, described in \ref{constG}, satisfies the following conditions:
\begin{itemize}
\item [\textup{(B1)}] $d_{i+1}(\alpha) \approx d_{i}(\alpha)g_{i+1}(\alpha)/g_{i}(\alpha);$
\item [\textup{(B2)}] $d_{i+1}(v_{i+1}) \approx d_{i}(\alpha)/g_{i}(\alpha);$
\item [\textup{(B3)}] $m_{i+1}(\alpha,v^2) \approx m_{i}(\alpha,v^2)g_{i+1}(\alpha)/g_{i}(\alpha)$ for each $v\in V_i\backslash\{\alpha\};$
\item [\textup{(B4)}] $m_{i+1}(v_{i+1},v^2) \approx m_{i}(\alpha,v^2)/g_i(\alpha)$ for each $v\in V_i\backslash\{\alpha\};$
\item [\textup{(B5)}] $m_{i+1}(\alpha, u, v) \approx m_{i}(\alpha, u, v)g_{i+1}(\alpha)/g_{i}(\alpha)$ for every pair of distinct vertices $u,v\in V_i\backslash\{\alpha\};$
\item [\textup{(B6)}] $m_{i+1}(v_{i+1}, u, v) \approx m_{i}(\alpha, u, v)/g_i(\alpha)$ for every pair of distinct vertices $u,v\in V_i\backslash\{\alpha\};$
\item [\textup{(B7)}] $m_{i+1}(v^2_{i+1}, v)=0$ for each $v\in V_i\backslash\{\alpha\};$
\item [\textup{(B8)}] $m_{i+1}(\alpha, v_{i+1}, v) \approx 2m_{i}(\alpha^2, v)/g_i(\alpha)$ for each $v\in V_i\backslash\{\alpha\};$
\item [\textup{(B9)}] $m_{i+1}(\alpha^2, v) \approx m_{i}(\alpha^2, v)(g_{i+1}(\alpha)-1)/g_i(\alpha)$ for each $v\in V_i\backslash\{\alpha\};$
\item [\textup{(B10)}] $m_{i+1}(v_{i+1}^3)=m_{i+1}(v^2_{i+1}, \alpha)=0;$
\item [\textup{(B11)}] $m_{i+1}(\alpha^3) \approx m_{i}(\alpha^3)(g_{i+1}(\alpha)-2)/g_{i}(\alpha);$
\item [\textup{(B12)}] $m_{i+1}(v_{i+1}, \alpha^2) \approx 3m_{i}(\alpha^3)/g_i(\alpha);$
\end{itemize}

\begin{proof} Since $H(\scr F_{i},\alpha)\in \scr A_i$, from (\ref{hinge1}) it follows that
\begin{eqnarray*}
d_{i+1}(v_{i+1})& = & |H(\scr F_{i+1},v_{i+1})|=|Z_i|=|Z_i\cap H(\scr F_{i},\alpha)|\\
& \approx & |H(\scr F_{i},\alpha)|/g_i(\alpha)=d_{i}(\alpha)/g_{i}(\alpha),\\
d_{i+1}(\alpha) & = &   |H(\scr F_{i+1},\alpha)|= |H(\scr F_{i},\alpha)|-|Z_i| \nonumber \\
& \approx&  d_{i}(\alpha)-d_{i}(\alpha)/g_{i}(\alpha)=(g_i(\alpha)-1)d_{i}(\alpha)/g_{i}(\alpha)\nonumber \\
&= &  d_{i}(\alpha)g_{i+1}(\alpha)/g_{i}(\alpha).
\end{eqnarray*}
This proves (B1) and (B2). 

If $v \in V_i\backslash\{\alpha\}$, then $H_i^{vv}\in \scr B_i$ and so 
\begin{eqnarray*}
m_{i+1}(v_{i+1}, v^2)& = & |Z_i\cap H_{i}^{vv}|\approx |H_{i}^{vv}| /g_i(\alpha)=m_{i}(\alpha, v^2)/g_{i}(\alpha),\\
m_{i+1}(\alpha,v^2) & = &   |H_{i}^{vv}|-|Z_i\cap H_{i}^{vv}|\approx m_{i}(\alpha, v^2)-m_{i}(\alpha, v^2)/g_{i}(\alpha)\nonumber \\
& =&  (g_i(\alpha)-1)m_{i}(\alpha, v^2)/g_{i}(\alpha)\\
&=&m_{i}(\alpha, v^2)g_{i+1}(\alpha)/g_{i}(\alpha).
\end{eqnarray*}
This proves (B3) and (B4) (see Figure \ref{figure:detachpossiblities}(i)). 
\begin{figure}[htbp]
\begin{center}
\scalebox{.75}
{ \includegraphics {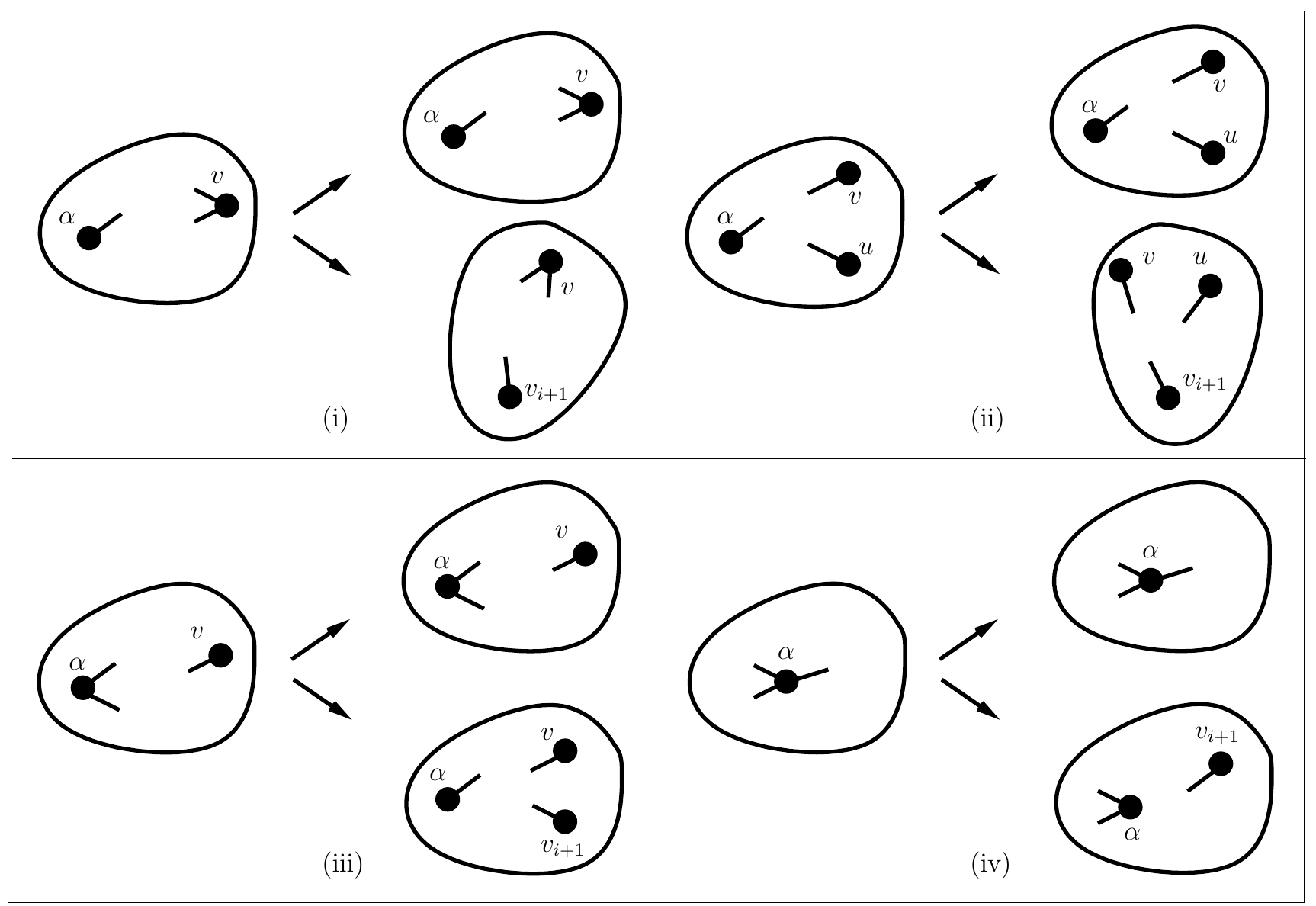} }
\caption{The four possibilities for detachment of a single edge incident with $\alpha$}
\label{figure:detachpossiblities}
\end{center}
\end{figure} 

If $u,v$ are a pair of distinct vertices in $V_i\backslash\{\alpha\}$, then $H_i^{uv}\in \scr B_i$ and so 
\begin{eqnarray*}
m_{i+1}(v_{i+1}, u, v)& = & |Z_i\cap H_{i}^{uv}|\approx |H_{i}^{uv}| /g_i(\alpha)=m_{i}(\alpha, u, v)/g_{i}(\alpha),\\
m_{i+1}(\alpha,u, v) & = &   |H_{i}^{uv}|-|Z_i\cap H_{i}^{uv}|\\
&\approx & m_{i}(\alpha, u, v)-m_{i}(\alpha, u, v)/g_{i}(\alpha)\nonumber \\
&= &  (g_i(\alpha)-1)m_{i}(\alpha, u, v)/g_{i}(\alpha)\\
&=&m_{i}(\alpha, u, v)g_{i+1}(\alpha)/g_{i}(\alpha).
\end{eqnarray*}
This proves (B5) and (B6) (see Figure \ref{figure:detachpossiblities}(ii)).  

If $v\in V_i\backslash\{\alpha\}$, and $e\in \nabla_{\scr F_i(j)}(\alpha^2, v)$, then $H_{ij}^{e}\in \scr A_i$, so 
$$\left |Z_i\cap H_{ij}^{e}\right |\approx |H_{ij}^{e}|/g_i(\alpha)=2/g_i(\alpha)\leq 1.$$
 Therefore either $|Z_i\cap H_{ij}^{e}|=1$ and consequently $e\in \nabla_{\scr F_{i+1}}(v_{i+1},\alpha, v)$ or $Z_i\cap H_{ij}^{e}=\varnothing$ and consequently $e\in \nabla_{\scr F_{i+1}} (\alpha^2,v)$. Therefore $$\nabla_{\scr F_{i+1}}(v^2_{i+1}, v)=\varnothing.$$ This proves (B7) (see Figure \ref{figure:detachpossiblities}(iii)). Moreover, since $H_{i}^{\alpha v}\in \scr B_i$
\begin{eqnarray*}
m_{i+1}(\alpha, v_{i+1}, v)& = & |Z_i\cap H_{i}^{\alpha v}|\approx |H_{i}^{\alpha v}| /g_i(\alpha)=2m_{i}(\alpha^2, v)/g_{i}(\alpha),\\
m_{i+1}(\alpha^2,v) & = &   m_{i}(\alpha^2, v)-|Z_i\cap H_{i}^{\alpha v}|\\
&\approx & m_{i}(\alpha^2, v)-2m_{i}(\alpha, u, v)/g_{i}(\alpha)\nonumber \\
& =&  (g_i(\alpha)-2)m_{i}(\alpha^2, v)/g_{i}(\alpha)\\
&=&m_{i}(\alpha^2, v)(g_{i+1}(\alpha)-1)/g_{i}(\alpha).
\end{eqnarray*}
This proves (B8) and (B9). We note that from (B9) it follows that if $g_{i+1}(\alpha)= 1$,  then $m_{i+1}(\alpha^2,v) =0$.

If $e$ is a loop in $\scr F_i(j)$ incident with $\alpha$, (so $g_i(\alpha)\geq 3$,) then $H_{ij}^{e}\in \scr A_i$. So 
$$|Z_i\cap H_{ij}^{e}|\approx |H_{ij}^{e}|/g_i(\alpha)=3/g_i(\alpha)\leq 1.$$ Therefore either $|Z_i\cap H_{ij}^{e}|=1$ and consequently $e\in \nabla_{\scr F_{i+1}}(\alpha^2, v_{i+1})$ or $Z_i\cap H_{ij}^{e}=\varnothing$ and consequently $e\in \nabla_{\scr F_{i+1}}(\alpha^3)$. Therefore 
$$\nabla_{\scr F_{i+1}}(v_{i+1}^3)=\nabla_{\scr F_{i+1}}(v^2_{i+1}, \alpha)=\varnothing.$$ This proves (B10) (see Figure \ref{figure:detachpossiblities}(iv)). Moreover,  
\begin{eqnarray*}
m_{i+1}(\alpha^2, v_{i+1})& = & |Z_i\cap H_{i}^{\alpha\alpha}|\approx |H_{i}^{\alpha\alpha}| /g_i(\alpha)=3m_{i}(\alpha^3)/g_{i}(\alpha),\\
m_{i+1}(\alpha^3) & = &  m_{i}(\alpha^3)-|Z_i\cap H_{i}^{\alpha\alpha}|\approx m_{i}(\alpha^3)-3m_{i}(\alpha^3)/g_{i}(\alpha)\nonumber \\
& =&  (g_i(\alpha)-3)m_{i}(\alpha^3)/g_{i}(\alpha)=m_{i}(\alpha^3)(g_{i+1}(\alpha)-2)/g_{i}(\alpha).
\end{eqnarray*}
This proves (B11) and (B12). We may note that from (B11) it follows that if $g_{i+1}(\alpha)= 2$,  then $m_{i+1}(\alpha^3)=0$. 
\end{proof}
A similar statement can be proved for every color class: Let us fix $j\in \{1,\ldots,k\}$, and let $u,v$ be a pair of distinct vertices in $V_i\backslash\{\alpha \}$. The colored version of (B7) and (B10) is trivial. Since $H_{ij}\in \scr A_i$, $H_{ij}^{vv}\in \scr B_i$, $H_{ij}^{uv}\in \scr B_i$, $H_{ij}^{\alpha v}\in \scr B_i$, $H_{ij}^{\alpha\alpha}\in \scr B_i$, respectively, we can obtain a colored version for (B1) and (B2), (B3) and (B4), (B5) and (B6), (B8) and (B9), and (B11) and (B12), respectively. 

\subsection{Relations between $\scr F_{i}$ and $\scr F$}\label{FiF}
Recall that $\Psi_i=\Phi_0\ldots\Phi_i$, that $\Phi_0:V\rightarrow V$, and that $\Phi_i:V_i\rightarrow V_{i-1}$ for $i>0$. Therefore $\Psi_i:V_i\rightarrow V$ and thus $\Psi_i^{-1}:V\rightarrow V_i$. 

Now we use (B1)--(B12) to prove that the hypergraph $\scr F_{i}$ satisfies the following conditions for $0\leq i \leq n:$
\begin{itemize}
\item [\textup{(D1)}] $d_{i}(x)/g_{i}(x) \approx d(x)/g(x)$ for each $x\in V;$
\item [\textup{(D2)}] $d_{i}(v_r) \approx d(x)/g(x)$ for each $x\in V$ and each $v_r\in \Psi_i^{-1}[x];$
\item [\textup{(D3)}] $m_{i}(x^3)/\binom{g_i(x)}{3} \approx m(x^3)/\binom{g(x)}{3}$ for each $x\in V$ with $g(x)\geq 3$ if $g_i(x)\geq 3$, and $m_{i}(x^3)=0$ otherwise;
\item [\textup{(D4)}] $m_{i}(v_r^3) =0$ for each $x\in V$ and each $v_r\in \Psi_i^{-1}[x];$
\item [\textup{(D5)}] $m_{i}(x^2, v_r)/\binom{g_i(x)}{2} \approx m(x^3)/\binom{g(x)}{3}$ for each $x\in V$ with $g(x)\geq 3$ and each $v_r\in \Psi_i^{-1}[x]$ if $g_i(x)\geq 2$, and $m_{i}(x^2, v_r)=0$ otherwise;
\item [\textup{(D6)}] $m_{i}(x, v_r, v_s)/g_i(x) \approx m(x^3)/\binom{g(x)}{3}$ for each $x\in V$ with $g(x)\geq 3$ and every pair of distinct vertices $v_r,v_s\in \Psi_i^{-1}[x];$
\item [\textup{(D7)}] $m_{i}(v_r, v_s,v_t) \approx m(x^3)/\binom{g(x)}{3}$ for each $x\in V$ with $g(x)\geq 3$ and every triple of distinct vertices $v_r,v_s,v_t\in \Psi_i^{-1}[x];$
\item [\textup{(D8)}] $m_{i}(x^2, y)/(\binom{g_{i}(x)}{2}g_{i}(y)) \approx m(x^2,y)/(\binom{g(x)}{2}g(y))$ for every pair of distinct vertices $x,y\in V$ with $g(x)\geq 2$ if $g_i(x)\geq 2$, and $m_{i}(x^2, y)=0$ otherwise;
\item [\textup{(D9)}] $m_{i}(x^2, v_t)/\binom{g_{i}(x)}{2} \approx m(x^2,y)/(\binom{g(x)}{2}g(y))$ for every pair of distinct vertices $x,y\in V$ with $g(x)\geq 2$ and each $v_t\in \Psi_i^{-1}[y]$ if $g_i(x)\geq 2$, and $m_{i}(x^2, v_t)=0$ otherwise;
\item [\textup{(D10)}] $m_{i}(x, v_r,y)/(g_i(x)g_i(y)) \approx m(x^2,y)/(\binom{g(x)}{2}g(y))$ for every pair of distinct vertices $x,y\in V$ with $g(x)\geq 2$ and each $v_r\in \Psi_i^{-1}[x];$
\item [\textup{(D11)}] $m_{i}(x,v_r, v_t) /g_i(x)\approx m(x^2,y)/(\binom{g(x)}{2}g(y))$  for every pair of distinct vertices $x,y\in V$ with $g(x)\geq 2$, each $v_r\in \Psi_i^{-1}[x]$ and each $v_t\in \Psi_i^{-1}[y];$
\item [\textup{(D12)}] $m_{i}(v_r, v_s,y)/g_i(y) \approx m(x^2,y)/(\binom{g(x)}{2}g(y))$ for every pair of distinct vertices $x,y\in V$ with $g(x)\geq 2$ and every pair of distinct vertices $v_r,v_s\in \Psi_i^{-1}[x];$
\item [\textup{(D13)}] $m_{i}(v_r, v_s,v_t) \approx m(x^2,y)/(\binom{g(x)}{2}g(y))$ for every pair of distinct vertices $x,y\in V$ with $g(x)\geq 2$, every pair of distinct vertices $v_r,v_s\in \Psi_i^{-1}[x]$ and each $v_t\in \Psi_i^{-1}[y];$
\item [\textup{(D14)}] $m_{i}(x, y,z)/(g_i(x)g_i(y)g_i(z)) \approx m(x,y,z)/(g(x)g(y)g(z))$ for every triple of distinct vertices $x,y,z\in V;$ 
\item [\textup{(D15)}] $m_{i}(x, y,v_t)/(g_i(x)g_i(y)) \approx m(x,y,z)/(g(x)g(y)g(z))$ for every triple of distinct vertices $x,y,z\in V$ and  each $v_t\in \Psi_i^{-1}[z];$
\item [\textup{(D16)}] $m_{i}(x, v_s,v_t)/g_i(x) \approx m(x,y,z)/(g(x)g(y)g(z))$ for every triple of distinct vertices $x,y,z\in V$, each $v_s\in \Psi_i^{-1}[y]$ and each $v_t\in \Psi_i^{-1}[z];$
\item [\textup{(D17)}] $m_{i}(v_r, v_s,v_t) \approx m(x,y,z)/(g(x)g(y)g(z))$ for every triple of distinct vertices $x,y,z\in V$, each $v_r\in \Psi_i^{-1}[x]$, each $v_s\in \Psi_i^{-1}[y]$ and  each $v_t\in \Psi_i^{-1}[z];$
\end{itemize}
\begin{proof}

Let $x,y,z$ be an arbitrary triple of distinct vertices of $V$. We prove (D1)--(D17) by induction. To verify (D1)--(D17) for $i=0$, recall that $\scr F_0=\scr F$, and $g _0(x)=g(x)$.

Obviously $d_{0}(x)/ g _0(x) = d(x)/g (x)$, and this proves (D1) for $i=0$. (D2) is trivial. If $g(x)\geq 3$, obviously $m_{0}(x^3)/\binom {g _0(x)}{3}  =  m(x^3)/\binom {g(x)}{3}$, and if $g(x)\leq 2$, by hypothesis of Theorem \ref{hyp1}, $m (x^3) = 0$. This proves (D3) for $i=0$. The proof of (D4)--(D17) for $i=0$ is similar and can be verified easily.

Now we will show that if $\scr F_i$ satisfies the conditions (D1)--(D17) for some $i< n$, then $\scr F_{i+1}$ (formed from $\scr F_i$ by detaching $v_{i+1}$ from the vertex $\alpha$) satisfies these conditions by replacing $i$ with $i+1$; we denote the corresponding conditions for $\scr F_{i+1}$ by (D1)$'$--(D17)$'$. If $g_{i+1}(x)=g_i(x)$, then (D1)$'$--(D7)$'$ are obviously true. So we just check  (D1)$'$--(D7)$'$ in the case where $x=\alpha$. 
Also if $g_{i+1}(x)=g_i(x)$ and $g_{i+1}(y)=g_i(y)$, then (D8)$'$--(D13)$'$ are clearly true. So in order to prove (D8)$'$--(D13)$'$, we shall assume that either $g_{i+1}(x)=g_i(x)-1$ or $g_{i+1}(y)=g_i(y)-1$ (so $\alpha \in \{x,y\}$). Similarly, if $g_{i+1}(x)=g_i(y),g_{i+1}(y)=g_i(y)$, and $g_{i+1}(z)=g_i(z)$, then (D14)$'$--(D17)$'$ are true.  Therefore to prove (D14)$'$--(D17)$'$ we shall assume that either $g_{i+1}(x)=g_i(x)-1$ or $g_{i+1}(y)=g_i(y)-1$ or $g_{i+1}(z)=g_i(z)-1$ (so $\alpha \in \{x,y,z\}$).  

\begin{itemize}
\item [\textup{(D1)$'$}] By (B1),  $d_{i+1}(\alpha)/g_{i+1}(\alpha) \approx d_{i}(\alpha)/g_{i}(\alpha)$, and by (D1) of the induction hypothesis $d_{i}(\alpha)/g_{i}(\alpha) \approx d(\alpha)/g(\alpha)$. Therefore 
$$\frac{d_{i+1}(\alpha)}{g_{i+1}(\alpha)}  \mathop\approx \limits^{\textup{(B1)}} \frac{d_{i}(\alpha)}{g_{i}(\alpha)} \mathop\approx \limits^{\textup{(D1)}}\frac{d(\alpha)}{g(\alpha)}.$$
This proves (D1)$'$. 
\item [\textup{(D2)$'$}] By (B2),  $d_{i+1}(v_{i+1}) \approx d_{i}(\alpha)/g_{i}(\alpha)$, and by (D1) of the induction hypothesis $d_{i}(\alpha)/g_{i}(\alpha) \approx d(\alpha)/g(\alpha)$. Therefore 
$$d_{i+1}(v_{i+1})  \mathop\approx \limits^{\textup{(B2)}} \frac{d_{i}(\alpha)}{g_{i}(\alpha)}\mathop\approx \limits^{\textup{(D1)}} \frac{d(\alpha)}{g(\alpha)}.$$
Since in forming $\scr F_{i+1}$ no hyperedge is detached from $v_r$ for each $v_r \in \Psi_i^{-1}[\alpha]$, we have $d_{i+1}(v_{r}) =d_{i}(v_{r})$. By (D2) of the induction hypothesis $d_{i}(v_{r})\approx d(\alpha)/g(\alpha)$ for each $v_r \in \Psi_i^{-1}[\alpha]$. Therefore
$$d_{i+1}(v_{r})=d_{i}(v_{r}) \mathop\approx \limits^{\textup{(D2)}}  \frac{d(\alpha)}{g(\alpha)}$$
for each $v_r \in \Psi_i^{-1}[\alpha]$. This proves (D2)$'$. 
\item [\textup{(D3)$'$}] Suppose $g(\alpha)\geq 3$. If $g_{i+1}(\alpha)\geq 3$, by (B11) 
\begin{eqnarray*}
\frac{m_{i+1}(\alpha^3)}{\binom{g_{i+1}(\alpha)}{3}} & \mathop\approx \limits^{\textup{(B11)}} &  \frac{m_{i}(\alpha^3)(g_{i+1}(\alpha)-2)}{g_{i}(\alpha)\binom{g_{i+1}(\alpha)}{3}} \\
& = & \frac{m_{i}(\alpha^3)(g_{i+1}(\alpha)-2)}{g_{i}(\alpha)g_{i+1}(\alpha)(g_{i+1}(\alpha)-1)(g_{i+1}(\alpha)-2)/6} \\
& = & \frac{m_{i}(\alpha^3)}{\binom{g_i(\alpha)}{3}}. 
\end{eqnarray*}
Since $g_{i}(\alpha)\geq 4>3$, by (D3) of the induction hypothesis $m_{i}(\alpha^3)/\binom{g_i(\alpha)}{3} \approx m(\alpha^3)/\binom{g(\alpha)}{3}$. Therefore
$$\frac{m_{i+1}(\alpha^3)}{\binom{g_{i+1}(\alpha)}{3}} \mathop\approx \limits^{\textup{(B11)}} \frac{m_{i}(\alpha^3)}{\binom{g_i(\alpha)}{3}}\mathop\approx \limits^{\textup{(D3)}}  \frac{m(\alpha^3)}{\binom{g(\alpha)}{3}}.$$
If $g_{i+1}(\alpha)< 3$, by (B11) $m_{i+1}(\alpha^3)=0$. This proves (D3)$'$. 
\item [\textup{(D4)$'$}] By (B10), $m_{i+1}(v_{i+1}^3)=0$. Moreover, $m_{i+1}(v_{r}^3)=m_{i}(v_{r}^3)=0$ for each $1\leq r\leq i$. This proves  (D4)$'$. 
\item [\textup{(D5)$'$}] Suppose $g(\alpha)\geq 3$. If $g_{i+1}(\alpha)\geq 2$, by (B12) 
\begin{eqnarray*}
\frac{m_{i+1}(\alpha^2, v_{i+1})}{\binom{g_{i+1}(\alpha)}{2} } & \mathop\approx \limits^{\textup{(B12)}} &  \frac{3m_{i}(\alpha^3)}{g_i(\alpha)\binom{g_{i+1}(\alpha)}{2}} \\
& = & \frac{3m_{i}(\alpha^3)}{g_{i}(\alpha)g_{i+1}(\alpha)(g_{i+1}(\alpha)-1)/2} \\
& = & \frac{m_{i}(\alpha^3)}{\binom{g_i(\alpha)}{3}}. 
\end{eqnarray*}
Since $g_{i}(\alpha)\geq 3$, by (D3) of the induction hypothesis $m_{i}(\alpha^3)/\binom{g_i(\alpha)}{3} \approx m(\alpha^3)/\binom{g(\alpha)}{3}$. Therefore
$$\frac{m_{i+1}(\alpha^2, v_{i+1})}{\binom{g_{i+1}(\alpha)}{2} }\mathop\approx \limits^{\textup{(B12)}} \frac{m_{i}(\alpha^3)}{\binom{g_i(\alpha)}{3}}  \mathop\approx \limits^{\textup{(D3)}} \frac{m(\alpha^3)}{\binom{g(\alpha)}{3}}.$$
By (B9) for each $v_r \in \Psi_i^{-1}[\alpha]$
\begin{eqnarray*}
\frac{m_{i+1}(\alpha^2, v_{r})}{\binom{g_{i+1}(\alpha)}{2} } & \mathop\approx \limits^{\textup{(B9)}} &  \frac{m_{i}(\alpha^2, v_r)(g_{i+1}(\alpha)-1)}{g_i(\alpha)\binom{g_{i+1}(\alpha)}{2}} \\
& = & \frac{m_{i}(\alpha^2, v_r)(g_{i+1}(\alpha)-1)}{g_i(\alpha)g_{i+1}(\alpha)(g_{i+1}(\alpha)-1)/2} \\
& = & \frac{m_{i}(\alpha^2, v_r)}{\binom{g_i(\alpha)}{2}}. 
\end{eqnarray*}
Since $g_{i}(\alpha)\geq 3>2$, by (D5) of the induction hypothesis we have $m_{i}(\alpha^2, v_r)/\binom{g_i(\alpha)}{2} \approx m(\alpha^3)/\binom{g(\alpha)}{3}$ for each $v_r\in \Psi_i^{-1}[\alpha]$. Therefore  
$$\frac{m_{i+1}(\alpha^2, v_{r})}{\binom{g_{i+1}(\alpha)}{2}} \mathop\approx \limits^{\textup{(B9)}} \frac{m_{i}(\alpha^2, v_r)}{\binom{g_i(\alpha)}{2}} \mathop\approx \limits^{\textup{(D5)}}\frac{m(\alpha^3)}{\binom{g(\alpha)}{3}}$$
for each $v_r\in \Psi_i^{-1}[\alpha]$. If $g_{i+1}(\alpha)=1$, by (B9) it follows that $m_{i+1}(\alpha^2, v_{r})=0$ for each $v_r \in \Psi_{i+1}^{-1}[\alpha]$. This proves (D5)$'$.
\item [\textup{(D6)$'$}] Suppose $g(\alpha)\geq 3$ and $v_r,v_s$ are a pair of distinct vertices in $ \Psi_i^{-1}[\alpha]$. From (B5) it follows that $m_{i+1}(\alpha, v_r, v_s)/g_{i+1}(\alpha) \approx m_{i}(\alpha, v_r, v_s)/g_i(\alpha)$. By (D6) of the induction hypothesis $m_{i}(\alpha, v_r, v_s)/g_i(\alpha) \approx m(\alpha^3)/\binom{g(\alpha)}{3}$. Therefore 
$$\frac{m_{i+1}(\alpha, v_r, v_s)}{g_{i+1}(\alpha)} \mathop\approx \limits^{\textup{(B5)}} \frac{m_{i}(\alpha, v_r, v_s)}{g_i(\alpha)} \mathop\approx \limits^{\textup{(D6)}}\frac{m(\alpha^3)}{\binom{g(\alpha)}{3}}.$$ 
From (B8) it follows that 
\begin{eqnarray*}
\frac{m_{i+1}(\alpha, v_r, v_{i+1})}{g_{i+1}(\alpha)}   \mathop\approx \limits^{\textup{(B8)}}   \frac{2m_{i}(\alpha^2, v_r)}{g_i(\alpha)g_{i+1}(\alpha)} =\frac{m_{i}(\alpha^2, v_r)}{\binom{g_i(\alpha)}{2} }.
\end{eqnarray*}
By (D5) of the induction hypothesis $m_{i}(\alpha^2, v_r)/\binom{g_i(\alpha)}{2} \approx m(\alpha^3)/\binom{g(\alpha)}{3}$. Therefore 
$$\frac{m_{i+1}(\alpha, v_r, v_{i+1})}{g_{i+1}(\alpha)}  \mathop\approx \limits^{\textup{(B8)}}  \frac{m_{i}(\alpha^2, v_r)}{\binom{g_i(\alpha)}{2} }\mathop\approx \limits^{\textup{(D5)}} \frac{m(\alpha^3)}{\binom{g(\alpha)}{3}}.$$
This proves (D6)$'$.
\item [\textup{(D7)$'$}] Suppose $g(\alpha)\geq 3$ and $v_r,v_s,v_t$ are a triple of distinct vertices in $ \Psi_i^{-1}[\alpha]$. Since in forming $\scr F_{i+1}$ no hyperedge is detached from $v_r, v_s, v_t$, we have $m_{i+1}(v_r, v_s,v_t) = m_{i}(v_r, v_s,v_t)$. But by (D7) of the induction hypothesis, $m_{i}(v_r, v_s,v_t)\approx  m(\alpha^3)/\binom{g(\alpha)}{3}$. Therefore 
$$m_{i+1}(v_r, v_s,v_t) =m_{i}(v_r, v_s,v_t)  \mathop\approx \limits^{\textup{(D7)}}  \frac{m(\alpha^3)}{\binom{g(\alpha)}{3}}.$$
By (B6) $m_{i+1}(v_r, v_s,v_{i+1}) \approx m_{i}(\alpha, v_r,v_s)/g_i(\alpha)$. By (D6) of the induction hypothesis $m_{i}(\alpha, v_r,v_s)/g_i(\alpha) \approx  m(\alpha^3)/\binom{g(\alpha)}{3}$. Therefore 
$$m_{i+1}(v_r, v_s,v_{i+1})   \mathop\approx \limits^{\textup{(B6)}} \frac{m_{i}(\alpha, v_r,v_s)}{g_i(\alpha)}\mathop\approx \limits^{\textup{(D6)}}  \frac{m(\alpha^3)}{\binom{g(\alpha)}{3}}.$$
This proves (D7)$'$. 
\item [\textup{(D8)$'$}] \textbf{Case 1:} If $g_{i+1}(x)=g_i(x)-1$ (so $x=\alpha$),  by (B9) $m_{i+1}(\alpha^2, y) \approx m_{i}(\alpha^2, y)(g_{i+1}(\alpha)-1)/g_i(\alpha)$ which is $0$ if $g_{i+1}(\alpha)=1$. If $g_{i+1}(\alpha)\geq 2$, by (B9)
\begin{eqnarray*}
\frac{m_{i+1}(\alpha^2, y)}{\binom{g_{i+1}(\alpha)}{2}g_{i+1}(y)} & \mathop\approx \limits^{\textup{(B9)}} &  \frac{m_{i}(\alpha^2, y)(g_{i+1}(\alpha)-1)}{g_i(\alpha)\binom{g_{i+1}(\alpha)}{2}g_{i+1}(y)} \\
& = & \frac{m_{i}(\alpha^2, y)(g_{i+1}(\alpha)-1)}{g_i(\alpha)g_{i+1}(\alpha)(g_{i+1}(\alpha)-1)g_i(y)/2} \\
& = & \frac{m_{i}(\alpha^2, y)}{\binom{g_{i}(\alpha)}{2}g_{i}(y)}. 
\end{eqnarray*}
Since $g_{i}(\alpha)\geq 3>2$, by (D8) of the induction hypothesis $m_{i}(\alpha^2, y)/(\binom{g_{i}(\alpha)}{2}g_{i}(y)) \approx m(\alpha^2,y)/(\binom{g(\alpha)}{2}g(y))$. Therefore 
$$\frac{m_{i+1}(\alpha^2, y)}{\binom{g_{i+1}(\alpha)}{2}g_{i+1}(y)} \mathop\approx \limits^{\textup{(B9)}} \frac{m_{i}(\alpha^2, y)}{\binom{g_{i}(\alpha)}{2}g_{i}(y)}\mathop\approx \limits^{\textup{(D8)}} \frac{m(\alpha^2,y)}{\binom{g(\alpha)}{2}g(y)}.$$
\textbf{Case 2:} If $g_{i+1}(y)=g_i(y)-1$ (so $y=\alpha$), by (B3) $m_{i+1}(x^2, \alpha) \approx m_{i}(x^2,\alpha)g_{i+1}(\alpha)/g_{i}(\alpha)$ which  is 0 by (D8) of the induction hypothesis, if $g_{i+1}(x)=g_i(x)=1$. If $g_{i+1}(x) \geq 2$, by (B3) and (D8) of the induction hypothesis
$$\frac{m_{i+1}(x^2, \alpha)}{\binom{g_{i+1}(x)}{2}g_{i+1}(\alpha)} \mathop  \approx \limits^{\textup{(B3)}}   \frac{m_{i}(x^2, \alpha)}{\binom{g_{i+1}(x)}{2}g_{i}(\alpha)}=\frac{m_{i}(x^2, \alpha)}{\binom{g_{i}(x)}{2}g_{i}(\alpha)}\mathop  \approx \limits^{\textup{(D8)}}  \frac{m(x^2,\alpha)}{\binom{g(x)}{2}g(\alpha)}.$$
This proves (D8)$'$.
\item [\textup{(D9)$'$}] Suppose $v_t\in \Psi_{i}^{-1}[y]$. There are two cases:\\
\textbf{Case 1:} If $g_{i+1}(x)=g_i(x)-1$ (so $x=\alpha$), by (B9) $m_{i+1}(\alpha^2, v_t) \approx  m_{i}(\alpha^2, v_t)(g_{i+1}(\alpha)-1)/g_i(\alpha)$ which is $0$ if  $g_{i+1}(\alpha)=1$. If $g_{i+1}(\alpha)\geq 2$, by (B9)
\begin{eqnarray*}
\frac{m_{i+1}(\alpha^2, v_t)}{\binom{g_{i+1}(\alpha)}{2} } &  \mathop\approx\limits^{\textup{(B9)}}  &  \frac{m_{i}(\alpha^2, v_t)(g_{i+1}(\alpha)-1)}{g_i(\alpha)\binom{g_{i+1}(\alpha)}{2}} \\
& = & \frac{m_{i}(\alpha^2, v_t)(g_{i+1}(\alpha)-1)}{g_i(\alpha)g_{i+1}(\alpha)(g_{i+1}(\alpha)-1)/2} \\
& = & \frac{m_{i}(\alpha^2, v_t)}{\binom{g_{i}(\alpha)}{2}}. 
\end{eqnarray*}
Since $g_{i}(\alpha)\geq 3>2$, by (D9) of the induction hypothesis we have $m_{i}(\alpha^2, v_t)/\binom{g_{i}(\alpha)}{2} \approx m(\alpha^2,y)/(\binom{g(\alpha)}{2}g(y))$. Therefore 
$$\frac{m_{i+1}(\alpha^2, v_t)}{\binom{g_{i+1}(\alpha)}{2}}  \mathop\approx\limits^{\textup{(B9)}}  \frac{m_{i}(\alpha^2, v_t)}{\binom{g_{i}(\alpha)}{2}} \mathop\approx\limits^{\textup{(D9)}}  \frac{m(\alpha^2,y)}{\binom{g(\alpha)}{2}g(y)}.$$
\textbf{Case 2:} If $g_{i+1}(y)=g_i(y)-1$ (so $y=\alpha$), since in forming $\scr F_{i+1}$ no hyperedge is detached from $v_t$ and $x$, we have $m_{i+1}(x^2, v_t) =m_{i}(x^2, v_t)$ which  is 0 by (D9) of the induction hypothesis, if $g_{i+1}(x)=g_i(x)=1$. If $g_{i+1}(x) \geq 2$, by (D9) of the induction hypothesis
$$\frac{m_{i+1}(x^2, v_t)}{\binom{g_{i+1}(x)}{2}} =\frac{m_{i}(x^2, v_t)}{\binom{g_{i}(x)}{2}}\mathop  \approx \limits^{\textup{(D9)}}  \frac{m(x^2,\alpha)}{\binom{g(x)}{2}g(\alpha)}.$$
By (B4), $m_{i+1}(v_{i+1},x^2) \approx m_{i}(\alpha,x^2)/g_i(\alpha)$ which is 0 by (D8) of the induction hypothesis, if $g_{i+1}(x)=g_i(x)=1$. If $g_{i+1}(x) \geq 2$, by (B4) and (D8) of the induction hypothesis  
$$\frac{m_{i+1}(x^2, v_{i+1})}{\binom{g_{i+1}(x)}{2}}  \mathop  \approx \limits^{\textup{(B4)}}   \frac{m_{i}(x^2, \alpha)}{\binom{g_{i+1}(x)}{2}g_i(\alpha)} 
= \frac{m_{i}(x^2, \alpha)}{\binom{g_{i}(x)}{2}g_i(\alpha)}  \mathop  \approx \limits^{\textup{(D8)}}  \frac{m_{i}(x^2, \alpha)}{\binom{g(x)}{2}g(\alpha)}. $$
This proves (D9)$'$.

\item [\textup{(D10)$'$}] Suppose $v_r\in \Psi_{i}^{-1}[x]$. There are two cases:\\
\textbf{Case 1:} If $g_{i+1}(x)=g_i(x)-1$ (so $x=\alpha$), by (B5) $m_{i+1}(\alpha, v_r,y)/g_{i+1}(\alpha)\approx m_{i}(\alpha, v_r,y)/g_{i}(\alpha)$. Therefore by (D10) of the induction hypothesis
\begin{eqnarray*}
\frac{m_{i+1}(\alpha, v_r,y)}{g_{i+1}(\alpha)g_{i+1}(y)} &  \mathop\approx\limits^{\textup{(B5)}}  &  \frac{m_{i}(\alpha, v_r,y)}{g_{i}(\alpha)g_{i+1}(y)} \\
& = & \frac{m_{i}(\alpha, v_r,y)}{g_{i}(\alpha)g_{i}(y)} \mathop  \approx \limits^{\textup{(D10)}} \frac{m(\alpha^2,y)}{\binom{g(\alpha)}{2}g(y)}. 
\end{eqnarray*} 
By (B8) $m_{i+1}(\alpha, v_{i+1},y)\approx 2m_{i}(\alpha^2, y)/g_i(\alpha)$. Therefore since $g_i(\alpha)\geq 2$, by (D8) of the induction hypothesis 
\begin{eqnarray*}
\frac{m_{i+1}(\alpha, v_{i+1},y)}{g_{i+1}(\alpha)g_{i+1}(y)} & \mathop  \approx \limits^{\textup{(B8)}}  &  \frac{2m_{i}(\alpha^2, y)}{g_i(\alpha)g_{i+1}(\alpha)g_{i+1}(y)} \\
& = & \frac{m_{i}(\alpha^2, y)}{\binom{g_i(\alpha)}{2}g_{i}(y)} \mathop  \approx \limits^{\textup{(D8)}} \frac{m(\alpha^2,y)}{\binom{g(\alpha)}{2}g(y)}. 
\end{eqnarray*} 
\textbf{Case 2:} If $g_{i+1}(y)=g_i(y)-1$ (so $y=\alpha$), by (B5) we have $m_{i+1}(x, v_r,\alpha)/g_{i+1}(\alpha) \approx m_{i}(x, v_r,\alpha)/g_{i}(\alpha)$. Therefore by (D10) of the induction hypothesis 
\begin{eqnarray*}
\frac {m_{i+1}(x, v_r,\alpha)}{g_{i+1}(x)g_{i+1}(\alpha)} & \mathop  \approx \limits^{\textup{(B5)}}  &  \frac{m_{i}(x, v_r,\alpha)}{g_{i+1}(x)g_{i}(\alpha)} \\
& = & \frac{m_{i}(x, v_r,\alpha)}{g_{i}(x)g_{i}(\alpha)} \mathop  \approx \limits^{\textup{(D10)}} \frac{m(x^2,\alpha)}{\binom{g(x)}{2}g(\alpha)}. 
\end{eqnarray*} 
This proves (D10)$'$.
\item [\textup{(D11)$'$}]  Suppose $v_r\in \Psi_{i}^{-1}[x],v_t\in \Psi_{i}^{-1}[y]$. There are two cases:\\
\textbf{Case 1:} If $g_{i+1}(x)=g_i(x)-1$ (so $x=\alpha$), by (B5) and (D11) of the induction hypothesis
$$\frac{m_{i+1}(\alpha,v_r, v_{t})}{g_{i+1}(\alpha)}  \mathop  \approx \limits^{\textup{(B5)}} \frac{m_{i}(\alpha,v_r, v_t) }{g_i(\alpha)} \mathop  \approx \limits^{\textup{(D11)}} \frac{m(\alpha^2,y)}{\binom{g(\alpha)}{2}g(y)}.$$
By (B8) $m_{i+1}(\alpha,v_{i+1}, v_{t}) \approx 2m_{i}(\alpha^2, v_t)/g_i(\alpha)$. Therefore by (D10) of the induction hypothesis
\begin{eqnarray*}
\frac {m_{i+1}(\alpha,v_{i+1}, v_{t})}{g_{i+1}(\alpha)} & \mathop  \approx \limits^{\textup{(B8)}}  &  \frac{2m_{i}(\alpha^2, v_t)}{g_i(\alpha)g_{i+1}(\alpha)} \\
& = & \frac{m_{i}(\alpha,v_r, y) }{\binom{g_i(\alpha)}{2}} \mathop  \approx \limits^{\textup{(D10)}} \frac{m(\alpha^2,y)}{\binom{g(\alpha)}{2}g(y)}. 
\end{eqnarray*} 
\textbf{Case 2:} If $g_{i+1}(y)=g_i(y)-1$ (so $y=\alpha$), since in forming $\scr F_{i+1}$ no hyperedge is detached from $x,v_r$ and $v_t$, we have $m_{i+1}(x,v_r, v_{t})= m_{i}(x,v_r, v_t)$. Therefore by (D11) of the induction hypothesis
$$\frac{m_{i+1}(x,v_r, v_{t}) }{g_{i+1}(x)} = \frac{m_{i}(x,v_r, v_t) }{g_{i+1}(x)}=\frac{m_{i}(x,v_r, v_t)}{g_{i}(x)}\mathop  \approx \limits^{\textup{(D11)}}  \frac{m(x^2,\alpha)}{\binom{g(x)}{2}g(\alpha)}.$$
By (B6) $m_{i+1}(v_{i+1},x,v_r) \approx m_{i}(\alpha,x,v_r) /g_i(\alpha)$. Therefore by (D10) of the induction hypothesis
$$\frac{m_{i+1}(x,v_r, v_{i+1})}{g_{i+1}(x)} \mathop  \approx \limits^{\textup{(B6)}} \frac{m_{i}(x,v_r, \alpha) }{g_{i+1}(x)g_i(\alpha)}=\frac{m_{i}(x,v_r, \alpha)}{g_i(x)g_i(\alpha)}\mathop  \approx \limits^{\textup{(D10)}} \frac{m(x^2,\alpha)}{\binom{g(x)}{2}g(\alpha)}.$$
This proves (D11)$'$.
\item [\textup{(D12)$'$}] Suppose $v_r,v_s\in \Psi_{i}^{-1}[x]$. There are two cases:\\
\textbf{Case 1:} If $g_{i+1}(x)=g_i(x)-1$ (so $x=\alpha$), since in forming $\scr F_{i+1}$ no hyperedge is detached from $v_r,v_s$ and $y$, we have $m_{i+1}(v_r, v_s,y)=m_{i}(v_r, v_s,y)$. Therefore by (D12) of the induction hypothesis
$$\frac{m_{i+1}(v_r, v_s,y)}{g_{i+1}(y)}=\frac{m_{i}(v_r, v_s,y)}{g_{i+1}(y)}=\frac{m_{i}(v_r, v_s,y)}{g_{i}(y)} \mathop  \approx \limits^{\textup{(D12)}} \frac{m(\alpha^2,y)}{\binom{g(\alpha)}{2}g(y)}.$$
By (B6) $m_{i+1}(v_{i+1},v_r, y) \approx m_{i}(\alpha, v_{r},y)/g_{i}(\alpha)$. Therefore by (D10) of the induction hypothesis
\begin{eqnarray*}
\frac{m_{i+1}(v_{i+1},v_r, y)}{g_{i+1}(y)}\mathop  \approx \limits^{\textup{(B6)}} \frac{m_{i}(\alpha, v_{r},y)}{g_{i}(\alpha)g_{i+1}(y)}=\frac{m_{i}(\alpha, v_{r},y)}{g_{i}(\alpha)g_{i}(y)}\mathop  \approx \limits^{\textup{(D10)}} \frac{m(\alpha^2,y)}{\binom{g(\alpha)}{2}g(y)}.
\end{eqnarray*}
\textbf{Case 2:} If $g_{i+1}(y)=g_i(y)-1$ (so $y=\alpha$), by (B5) and (D12) of the induction hypothesis 
$$\frac{m_{i+1}(v_r, v_s,\alpha)}{g_{i+1}(\alpha)}\mathop  \approx \limits^{\textup{(B5)}} \frac{m_{i}(v_r, v_s,\alpha)}{g_{i}(\alpha)} \mathop  \approx \limits^{\textup{(D12)}} \frac{m(x^2,\alpha)}{\binom{g(x)}{2}g(\alpha)}.$$
This proves (D12)$'$.
\item [\textup{(D13)$'$}] Suppose $v_r,v_s\in \Psi_{i}^{-1}[x],v_t\in \Psi_{i}^{-1}[y]$. Since in forming $\scr F_{i+1}$ no hyperedge is detached from $v_r,v_s$ and $v_t$, we have $m_{i+1}(v_r, v_s,v_t) =m_{i}(v_r, v_s,v_t)$. Therefore by (D13) of the induction hypothesis
 $$m_{i+1}(v_r, v_s,v_t) \mathop  \approx \limits^{\textup{(D13)}} \frac{m(x^2,y)}{\binom{g(x)}{2}g(y)}.$$
If $g_{i+1}(x)=g_i(x)-1$ (so $x=\alpha$),  by (B6) and (D11) of the induction hypothesis
$$m_{i+1}(v_r, v_{i+1},v_t)  \mathop  \approx \limits^{\textup{(B6)}}  \frac{m_{i}(\alpha, v_r, v_t)}{g_{i}(\alpha)} \mathop  \approx \limits^{\textup{(D11)}}  \frac{m(\alpha^2,y)}{\binom{g(\alpha)}{2}g(y)}.$$
If $g_{i+1}(y)=g_i(y)-1$ (so $y=\alpha$), by (B6) and (D12) of the induction hypothesis
$$m_{i+1}(v_r, v_s,v_{i+1}) \mathop  \approx \limits^{\textup{(B6)}} \frac{m_{i}(\alpha, v_r, v_s)}{g_{i}(y)}\mathop  \approx \limits^{\textup{(D12)}} \frac{m(x^2,\alpha)}{\binom{g(x)}{2}g(\alpha)}.$$
This proves (D13)$'$.
\item [\textup{(D14)$'$}]  If $g_{i+1}(x)=g_i(x)-1$ (so $x=\alpha$) , by (B5) $m_{i+1}(\alpha, y,z)/g_{i+1}(\alpha) \approx m_{i}(\alpha, y,z)/g_{i}(\alpha)$. Therefore by (D14) of the induction hypothesis
\begin{eqnarray*}
\frac{m_{i+1}(\alpha, y,z)}{g_{i+1}(\alpha)g_{i+1}(y)g_{i+1}(z)} &\mathop  \approx \limits^{\textup{(B5)}} &\frac{m_{i}(\alpha, y,z)}{g_{i}(\alpha)g_{i+1}(y)g_{i+1}(z)}\\
&=&\frac{m_{i}(\alpha, y,z)}{g_{i}(\alpha)g_{i}(y)g_{i}(z)}\mathop  \approx \limits^{\textup{(D14)}} \frac{m(\alpha,y,z)}{g(\alpha)g(y)g(z)}.
\end{eqnarray*}
There are two other cases ($g_{i+1}(y)=g_i(y)-1$ and $g_{i+1}(z)=g_i(z)-1$) for which the proof is similar. This proves (D14)$'$.
\item [\textup{(D15)$'$}] Suppose $v_t\in \Psi_{i}^{-1}[z]$. There are three cases:\\
\textbf{Case 1:} If $g_{i+1}(x)=g_i(x)-1$ (so $x=\alpha$) , by (B5) $m_{i+1}(\alpha, y,v_{t})/g_{i+1}(\alpha) \approx m_{i}(\alpha, y,v_t)/g_i(\alpha)$. Therefore by (D15) of the induction hypothesis
\begin{eqnarray*}
\frac{m_{i+1}(\alpha, y,v_t)}{g_{i+1}(\alpha)g_{i+1}(y)} &\mathop  \approx \limits^{\textup{(B5)}} &\frac{m_{i}(\alpha, y,v_t)}{g_{i}(\alpha)g_{i+1}(y)}\\
&=&\frac{m_{i}(\alpha, y,v_t)}{g_{i}(\alpha)g_{i}(y)}\mathop  \approx \limits^{\textup{(D15)}} \frac{m(\alpha,y,z)}{g(\alpha)g(y)g(z)}.
\end{eqnarray*}  
\textbf{Case 2:} If $g_{i+1}(y)=g_i(y)-1$ (so $y=\alpha$), the proof is similar to that of case 1.\\
\textbf{Case 3:} If $g_{i+1}(z)=g_i(z)-1$ (so $z=\alpha$), since in forming $\scr F_{i+1}$ no hyperedge is detached from $x,y$ and $v_t$, we have $m_{i+1}(x, y,v_t) \approx m_{i}(x, y,v_t)$. Therefore by (D15) of the induction hypothesis
$$\frac{m_{i+1}(x, y,v_t)}{g_{i+1}(x)g_{i+1}(y)} = \frac{m_{i}(x, y,v_t)}{g_i(x)g_i(y)} \mathop  \approx \limits^{\textup{(D15)}}  \frac{m(x,y, \alpha)}{g(x)g(y)g(\alpha)}.$$
By (B6) $m_{i+1}(x, y,v_{i+1}) \approx m_{i}(x, y, \alpha)/g_i(\alpha)$. Therefore by (D14) of the induction hypothesis
\begin{eqnarray*}
\frac{m_{i+1}(x, y,v_{i+1})}{g_{i+1}(x)g_{i+1}(y)} &\mathop  \approx \limits^{\textup{(B6)}} &\frac{m_{i}(x, y, \alpha)}{g_{i+1}(x)g_{i+1}(y)g_i(\alpha)}\\
&=&\frac{m_{i}(x, y, \alpha)}{g_{i}(x)g_{i}(y)g_i(\alpha)}\mathop  \approx \limits^{\textup{(D14)}} \frac{m(x,y, \alpha)}{g(x)g(y)g(\alpha)}.
\end{eqnarray*}  
This proves (D15)$'$.
\item [\textup{(D16)$'$}] Suppose $v_s\in \Psi_{i}^{-1}[y], v_t\in \Psi_{i}^{-1}[z]$. There are three cases:\\
\textbf{Case 1:} If $g_{i+1}(x)=g_i(x)-1$ (so $x=\alpha$) , by (B5) and (D16) of the induction hypothesis
$$\frac{m_{i+1}(\alpha, v_{s},v_t)}{g_{i+1}(\alpha)}\mathop  \approx \limits^{\textup{(B5)}} \frac{m_{i}(\alpha, v_s,v_t)}{g_i(\alpha)}\mathop  \approx \limits^{\textup{(D16)}} \frac{m(\alpha,y,z)}{g(\alpha)g(y)g(z)}.$$
\textbf{Case 2:} If $g_{i+1}(y)=g_i(y)-1$ (so $y=\alpha$), since in forming $\scr F_{i+1}$ no hyperedge is detached from $x, v_s$ and $v_t$, we have  $m_{i+1}(x, v_s,v_t)=m_{i}(x, v_s,v_t)$. Therefore by (D16) of the induction hypothesis
$$\frac{m_{i+1}(x, v_s,v_t)}{g_{i+1}(x)}=\frac{m_{i}(x, v_s,v_t)}{g_{i}(x)}\mathop  \approx \limits^{\textup{(D16)}} \frac{m(x, \alpha,z)}{g(x)g(\alpha)g(z)}.$$
By (B6) $m_{i+1}(x, v_{i+1},v_t)\approx m_{i}(x, \alpha,v_t)/g_i(\alpha)$. Therefore by (D15) of the induction hypothesis
\begin{eqnarray*}
\frac{m_{i+1}(x, v_{i+1},v_t)}{g_{i+1}(x)}& \mathop  \approx \limits^{\textup{(B6)}} &\frac{m_{i}(x, \alpha,v_t)}{g_{i+1}(x)g_i(\alpha)}\\
&=&\frac{m_{i}(x, \alpha,v_t)}{g_{i}(x)g_i(\alpha)}\mathop  \approx \limits^{\textup{(D15)}} \frac{m(x, \alpha,z)}{g(x)g(\alpha)g(z)}.
\end{eqnarray*}  
\textbf{Case 3:} If $g_{i+1}(z)=g_i(z)-1$ (so $z=\alpha$), the proof is similar to that of case 2. This proves (D16)$'$.
\item [\textup{(D17)$'$}] Suppose $v_r\in \Psi_{i}^{-1}[x], v_s\in \Psi_{i}^{-1}[y], v_t\in \Psi_{i}^{-1}[z]$. Since in forming $\scr F_{i+1}$ no hyperedge is detached from $v_r,v_s$ and $v_t$, we have $m_{i+1}(v_r, v_s,v_t)=m_{i}(v_r, v_s,v_t)$. Therefore by (D17) of the induction hypothesis
$$m_{i+1}(v_r, v_s,v_t) \mathop  \approx \limits^{\textup{(D17)}} \frac{m(x,y,z)}{g(x)g(y)g(z)}.$$
If $g_{i+1}(x)=g_i(x)-1$ (so $x=\alpha$) , by (B6) and (D16) of the induction hypothesis
$$m_{i+1}(v_{i+1}, v_s,v_t)  \mathop  \approx \limits^{\textup{(B6)}} \frac{m_{i}(\alpha, v_s,v_t)}{g_i(\alpha)}  \mathop  \approx \limits^{\textup{(D16)}} \frac{m(\alpha,y,z)}{g(\alpha)g(y)g(z)}.$$
There are two other cases ($g_{i+1}(y)=g_i(y)-1$ and $g_{i+1}(z)=g_i(z)-1$) for which the proof is similar. This proves (D17)$'$.
\end{itemize}
\end{proof}
A similar statement can be proved for every color class simply by restricting each relation above to a color class $j\in \{1,\dots, k\}$.

\subsection{Relations between $\scr G=\scr F_{n}$ and $\scr F$}\label{FnF}
Recall that $\scr G=\scr F_n, \Psi=\Psi_n$ and $g_n(x)=1$ for each $x\in V$. 
We claim that $\scr G$ satisfies all conditions stated in Theorem \ref{hyp1}. 

Obviously $\scr G$ is a $g$-detachment of $\scr F$. Let $x, y, z$ be an arbitrary triple of distinct vertices of $V$, and let $j\in \{1,\ldots,k\}$.  
Now in (D1)--(D17) we let $i=n$. From (D3) and (D4) it is immediate that $\scr G$ is loopless. From (D5), (D8) and (D9) it follows that $\scr G$ has no hyperedge of size 2. Thus $\scr G$ is a 3-uniform hypergraph. 

From (D1) it follows that $d_{\scr F_n}(x)/g_n(x) \approx d(x)/g (x)$, so  $d_{\scr G}(x) \approx d(x)/g(x)$. From (D2), $d_{\scr F_n}  (v_r) \approx d(x)/g (x) $ for each $v_r\in\Psi _n^{-1}[x]$, so $d_{\scr G}  (v_r) \approx d(x)/g(x) $ for each $v_r\in\Psi ^{-1}[x]$. Therefore $\scr G$ satisfies (A1). 

A similar argument shows that (A2) follows from the colored version of (D1) and (D2), (A3) follows from (D6), (D7), and (D10)--(D17), and (A4) follows from the colored version of (D6), (D7), and (D10)--(D17). This completes the proof of Theorem \ref{hyp1}.

\section{Algorithmic Aspects } \label{algocons}
To construct an $r$-factorization for $\lambda K_n^3$, we start with an amalgamation of $\lambda K_n^3$ in which all hyperedges are loops. We color the hyperedges among $k:=\lambda \binom{n-1}{2}/r$ color classes as evenly as possible, and apply Theorem \ref{hyp1}. In Theorem \ref{hyp1}, we detach vertices in $n-1$ steps. At each step, to decide how to share edges (and hinges) among the new vertices, we define two sets $\mathscr A$ and $\mathscr B$ whose sizes are no more than $1+k+\binom{n}{3}$ and $(k+1)\binom{n}{2}$, respectively, and use Nash-Willimas lemma. Nash-Williams lemma builds a graph of size $O(n^3)$ (or more precisely of size $|\mathscr A|+|\mathscr B|$) and finds a set $Z$ with a polynomial time algorithm.  The set $Z$ tells us exactly how to share edges (and hinges) among the new vertices. 
Therefore, our construction is polynomial in $\binom{n}{3}$, the output size for the problem.

\section{Acknowledgement}
The author wishes to thank his PhD supervisor professor Chris Rodger, his colleague Frank Sturm, and the referees for their valuable suggestions.

\bibliographystyle{model1-num-names}
\bibliography{<your-bib-database>}

\begin{thebibliography}{99}           
\footnotesize{            

\bibitem {AndHil1} L.D. Andersen, A.J.W. Hilton, Generalized Latin rectangles I: Construction and decomposition, Discrete Math. 31(2) (1980) 125--152.
\bibitem {AndHil2} L.D. Andersen,  A.J.W. Hilton, Generalized Latin rectangles II: Embedding, Discrete Math. 31(3) (1980), 235--260.
\bibitem{AndRod1} L.D. Andersen, C.A. Rodger, Decompositions of complete graphs: Embedding partial edge-colourings and the method of amalgamations, Surveys in Combinatorics, Lond Math Soc Lect Note Ser 307 (2003), 7--41.
\bibitem{BahThesis} M.A. Bahmanian, Amalgamations and Detachments of Graphs and Hypergraphs, PhD Dissertation, Auburn University, 2012. 
\bibitem{BahRod4Emb2}  M.A. Bahmanian, C.A. Rodger,  Extending partial edge-colorings of complete 3-uniform hypergraphs to $r$-factorizations, submitted for publication. 
\bibitem {Baran75} Zs. Baranyai, On the factorization of the complete uniform hypergraph, Infinite and finite sets (Colloq., Keszthely, 1973; dedicated to P. Erd\H os on his 60th birthday), Vol. I, pp. 91--108. Colloq. Math. Soc. Janos Bolyai, Vol. 10, North-Holland, Amsterdam, 1975.
\bibitem {Baran79} Zs. Baranyai, The edge-coloring of complete hypergraphs I, J. Combin. Theory  B 26(3) (1979), 276--294. 
\bibitem {HypBerge} C. Berge, Hypergraphs, North-Holland, Amsterdam, 1989.
\bibitem {BeJacksonJor} A.R. Berg, B. Jackson, T. Jord\'{a}n,  Highly edge-connected detachments of graphs and digraphs, J. Graph Theory 43 (2003) 67--77. 
\bibitem {BrouSchrij} A.E. Brouwer, A. Schrijver, Uniform hypergraphs, in: A. Schrijver (ed.), Packing and Covering in Combinatorics, Mathematical Centre Tracts 106, Amsterdam (1979), 39--73.
\bibitem {B} H. Buchanan II, Graph factors and Hamiltonian decompositions, Ph.D. Dissertation, West Virginia University 1997.
\bibitem {deW71}D. de Werra, Balanced schedules, INFOR---Canad. J. Operational Res. and Information Processing 9 (1971), 230--237. 
\bibitem {deW71-2} D. de Werra, Equitable colorations of graphs, Rev. Fran. Inf. Rech. Oper. 5 (1971), 3--8.
\bibitem {deW75}D. de Werra, A few remarks on chromatic scheduling, Combinatorial programming: methods and applications (Proc. NATO Advanced Study Inst., Versailles, 1974), pp. 337--342. NATO Advanced Study Inst. Ser., Ser. C: Math. and Phys. Sci., Vol. 19, Reidel, Dordrecht, 1975. 
\bibitem {deW75-2} D. de Werra, On a particular conference scheduling problem. INFOR---Canad. J. Operational Res. and Information Processing 13 (1975), no. 3, 308--315.  
\bibitem {H2} A.J.W. Hilton, Hamilton decompositions of complete graphs, J. Combin. Theory B 36 (1984), 125--134.
\bibitem {HJRW} A.J.W. Hilton, M. Johnson, C.A. Rodger, E.B. Wantland, Amalgamations of connected $k$-factorizations, J. Combin. Theory B 88 (2003)  267--279.
\bibitem {HR} A.J.W. Hilton, C.A. Rodger, Hamilton decompositions of complete regular $s$-partite graphs, Discrete Math. 58 (1986) 63--78.
\bibitem {JacksonJor} B. Jackson, T. Jord\'{a}n, Non-separable detachments of graphs, Dedicated to Crispin St. J. A. Nash-Williams, J. Combin. Theory Ser. B 87 (2003) 17--37. 
\bibitem {MatJohns} M. Johnson, Amalgamations of factorizations of complete graphs, J. Combin. Theory B 97 (2007), 597--611.
\bibitem {JungGNA} D. Jungnickel, Graphs, networks and algorithms, Third edition, Springer, Berlin, 2008. 
\bibitem{Kirk1847} T.P  Kirkman, On a problem in combinations, Camb. Dublin Math. J. 2 (1847) 191--204.
\bibitem {LA} R. Laskar, B. Auerbach, On the decompositions of $r$-partite graphs into edge-disjoint hamilton circuits,  Discrete Math. 14 (1976) 146--155.
\bibitem {LR1} C.D. Leach, C.A. Rodger, Non-disconnecting disentanglements of amalgamated 2-factorizations of complete multipartite graphs, J. Combin. Des. 9 (2001)  460--467.
\bibitem {LR2} C.D. Leach, C.A. Rodger, Hamilton decompositions of complete multipartite graphs with any 2-factor leave, J. Graph Theory 44 (2003) 208--214.
\bibitem {LR3} C.D. Leach, C.A. Rodger, Hamilton decompositions of complete graphs with a 3-factor leave, Discrete Math. 279 (2004) 337--344.
\bibitem {L} D.E. Lucas, R\'ecr\'eations Math\'ematiques, Vol. 2, Gauthiers Villars, Paris, 1892. 
\bibitem{NashW} C.St.J.A. Nash-Williams, Connected detachments of graphs and generalized Euler trails. J. London Math. Soc. 31 (1985) 17--29.
\bibitem{NashW85} C.St.J.A. Nash-Williams, Detachments of graphs and generalised Euler trails, in: I. Anderson (Ed.), Surveys in Combinatorics 1985, London Mathematical Society, Lecture Note Series, Vol. 103, Cambridge University Press, Cambridge, 1985, pp. 137--151.
\bibitem {Nash87} C.St.J.A. Nash-Williams, Amalgamations of almost regular edge-colourings of simple graphs, J. Combin. Theory B 43 (1987) 322--342.  
\bibitem {Pelt} R. Peltesohn, Das Turnierproblem f\"ur Spiele zu je dreien, Inaugural dissertation, Berlin, 1936. 
\bibitem{RayWil73} D.K. Ray-Chaudhuri, R.M. Wilson, The existence of resolvable designs, in A Survey of Combinatorial Theory, J. N. Srivastava et al. (Editors), North-Holland, Amsterdam, 1973, pp. 361--376.
\bibitem {RW} C.A. Rodger, E.B. Wantland, Embedding hyperedge-colorings into 2-edge-connected $k$-factorizations of $K_{kn+1}$, J. Graph Theory 10 (1995) 169--185.  
}
\end{thebibliography}

\end{document}